\documentclass[11pt]{amsart}

\usepackage{amsmath}
\usepackage{amssymb}
\usepackage{mathrsfs}
\usepackage{xcolor, enumitem, comment, todonotes}
\usepackage[all]{xy}
 \usepackage{url}
 \usepackage{amsthm}
 \usepackage{thmtools}
\usepackage{thm-restate}
\usepackage{hyperref,soul}

\theoremstyle{plain}
\newtheorem*{theorem*}{Theorem}

\newtheorem{theorem}{Theorem}[section]
\newtheorem{claim}[theorem]{Claim}

\newtheorem{lemma}[theorem]{Lemma}

\newtheorem{proposition}[theorem]{Proposition}
\newtheorem{corollary}[theorem]{Corollary}

\theoremstyle{definition}
\newtheorem{definition}[theorem]{Definition}
\newtheorem{example}[theorem]{Example}

\newtheorem{question}[theorem]{Question}

\theoremstyle{remark}
\newtheorem{remark}[theorem]{Remark}

\def\l{{\langle}}
\def\r{{\rangle}}

\newcount\skewfactor
\def\mathunderaccent#1#2 {\let\theaccent#1\skewfactor#2
\mathpalette\putaccentunder}
\def\putaccentunder#1#2{\oalign{$#1#2$\crcr\hidewidth
\vbox to.2ex{\hbox{$#1\skew\skewfactor\theaccent{}$}\vss}\hidewidth}}

\newcommand{\lusim}[1]{\smash{\underset{\raisebox{1.2pt}[0cm][0cm]{$\sim$}}
{{#1}}}}
\def\smallbox#1{\leavevmode\thinspace\hbox{\vrule\vtop{\vbox
   {\hrule\kern1pt\hbox{\vphantom{\tt/}\thinspace{\tt#1}\thinspace}}
   \kern1pt\hrule}\vrule}\thinspace}

\DeclareMathOperator{\dom}{dom}

\title[Diamond on Omega]{Diamond principles and Tukey-top ultrafilters on a countable set}
\date{\today}
\author{Tom Benhamou}
\thanks{The research of the first author was supported by the National Science Foundation under Grant
No. DMS-2346680}


\author{Fanxin Wu}
\address[Benhamou]{Department of Mathematics, Rutgers University, New Brunswick ,NJ USA}
\email{tom.benhamou@rutgers.edu}
\address[F. Wu]{Department of Mathematics, Rutgers University, New Brunswick ,NJ USA}
\email{fanxin.wu@rutgers.edu}

\begin{document}

\begin{abstract}
    We provide two types of guessing principles for ultrafilter ($\diamondsuit^{-}_{\lambda}(U), \ \diamondsuit^p_\lambda(U)$) on $\omega$ which form subclasses of Tukey-top ultrafilters, and construct such ultrafilters in $ZFC$. These constructions are essentially different from Isbell's construction  \cite{Isbell65} of Tukey-top ultrafilters. We prove using the Borel-Cantelli Lemma that full guessing is not possible and rule out several stronger guessing principles e.g. we prove that no Dodd-sound ultrafilters exist on $\omega$. We then apply these guessing principles to force a $q$-point which is Tukey-top (answering a question from \cite{TomNatasha}), and prove that the class of ultrafilters which satisfy $\neg\diamondsuit^{-}_\lambda$ is closed under Fubini sum. Finally, we show that $\diamondsuit^{-}_\lambda$ and $\diamondsuit^p_\lambda$ can be separated. 
\end{abstract}
\maketitle
\section{introduction}
The classical diamond principle $\diamondsuit$, introduced by Jensen \cite{Jensen1972}, says there is a sequence $\l A_\alpha\mid \alpha<\omega_1\r$ such that $A_\alpha\subseteq\alpha$ and for every set $X\subseteq \omega_1$, $\{\alpha<\omega_1\mid A_\alpha=X\cap\alpha\}$ is stationary. Intuitively, any subset of $\omega_1$ is ``guessed'' by the sequence $A_\alpha$ on a large set. It is well-known that $\diamondsuit$ holds in $L$, and that it implies $CH$. The diamond principle is extremely useful in inductive constructions of length $\omega_1$, where one ensures certain properties of the resulting mathematical object by using the sequence $A_\alpha$ to anticipate the obstruction. The classic example is the construction of a Suslin tree. Other examples include topological spaces, groups, Banach spaces, etc.

Jensen also formulated the analogous principle $\diamondsuit(\kappa)$ for any uncountable regular cardinal $\kappa$. For more on diamond-like principles, see the survey \cite{Rinot2010}. One might be tempted to try and formulate it for $\omega$ also. Unfortunately, the notions of club and stationarity only make sense on ordinals of uncountable cofinality, so a naive generalization to $\kappa=\omega$ doesn't work. The second try is to change the notion of ``largeness'' from ``stationary'' to  
``positive'' with respect to some filter $F$ on $\omega$, that is, to require that the set of $n$'s at which the sequence guesses correctly is positive with respect to $F$. However, an easy diagonalization argument rules this out: for suppose that $\l A_n\mid n<\omega\r$ is a sequence of sets such that $A_n\subseteq\{0,...,n-1\}$, we define $A=\{n-1\mid n-1\notin A_n\}$, then $\{n<\omega\mid A\cap \{0,...,n-1\}=A_n\}=\emptyset$.

Another plausible attempt is to generalize the weak diamond principle $\diamondsuit^-$ (also due to Jensen) instead of $\diamondsuit$. Recall that $\diamondsuit^-(\kappa)$ for $\kappa$ regular uncountable replaces the diamond sequence by a sequence $\l\mathcal{A}_\alpha\mid\alpha<\kappa\r$ such that $\mathcal{A}_\alpha\subseteq P(\alpha)$ and $|\mathcal{A}_\alpha|\leq\alpha$; The requirement $|\mathcal{A}_\alpha|\leq\alpha$ is (more or less) equivalent to $|\mathcal{A}_\alpha|<\kappa$ if $\kappa=\tau^+$ is a successor cardinal, but not for $\kappa$ inaccessible, since if we let $\mathcal{A}_\alpha=P(\alpha)$ then trivially every set is guessed everywhere. This suggests that a generalization of $\diamondsuit^-$ to $\omega$ should also involve a sequence $\l\mathcal{A}_n\mid n<\omega\r$ where we impose some restriction on the growth rate of $n\mapsto|\mathcal{A}_n|$, e.g., $|\mathcal{A}_n|\leq n$.

This still does not work: there is no hope of finding a non-principal filter $F$ for which there is a $\diamondsuit^-$ guessing sequence $\l\mathcal{A}_n\mid n<\omega\r$, to see this we apply the Borel-Cantelli Lemma \cite{Borel,Cantelli} from probability theory; see Proposition \ref{Thm: Borel Cantelli application} below. However, if instead of requiring the sequence to guess \textit{all} subsets of $\omega$, we only ask it to guess \textit{continuum many} subsets, then this turns out to be a meaningful definition.

In this paper, we would like to propose two diamond-like principles for (ultra)filters on $\omega$, which we call \textit{diamond minus} $\diamondsuit^-_{\mathfrak{c}}$ (Definition \ref{Def: diamond}) and \textit{perfect diamond} $\diamondsuit^{p}_{\mathfrak{c}}$ (Definition \ref{Def: Perfect Diamond}), and prove that in $ZFC$ we can construct (ultra)filters witnessing these principles (Theorem \ref{Thm: exist splitting tree}). The motivation for these principles is not so much trying to generalize diamond to $\omega$, as trying to generalize the following result that comes from a recent work of the first author with G. Goldberg \cite{TomGabe}, where similar principles on $\sigma$-complete ultrafilters over uncountable cardinals were formulated. 

\begin{theorem}[Benhamou-Goldberg]
   In all the known $ZFC$ canonical inner models \footnote{That is, inner models of the form $L[\mathbb{E}]$, where $\mathbb{E}$ is a sequence of extenders indexed by the Mitchell-Steel indexing system. More generally, the result holds in models of the ultrapower axiom where every irreducible ultrafilter is Dodd-sound (see \cite{GoldbergUA}).}, the following are equivalent for every $\sigma$-complete ultrafilter $U$ over a regular uncountable cardinal:
   \begin{enumerate}
       \item $U$ is Tukey-top i.e. there are $2^\kappa$-many sets $\l A_i\mid i<2^\kappa\r\subseteq U$ such that for every $I\in [2^\kappa]^\kappa$, $\bigcap_{i\in I}A_i\notin U$.
       \item $\diamondsuit^-_{2^\kappa}(U)$.
       \item $U$ is not isomorphic to an $n$-fold sum of $p$-points \footnote{$W$ is an $n$-fold sum of $p$-points if $W=\sum_{U}\sum_{U_{\alpha_1}}\sum...\sum_{U_{\alpha_1,...,\alpha_{n}}}U_{\alpha_1,...,\alpha_{n+1}}$, where each $U_{\alpha_1,...,\alpha_k}$, and $U$ are $p$-points}
   \end{enumerate}
\end{theorem}
The context of this result, is the renewed interest in the Tukey-top class of ultrafilters at the level of measurable cardinals, also known as the class of non-Galvin ultrafilters  \cite{AbrahamShelah1986,gitikGal,GalDet,partOne,Parttwo,Non-GalvinFil}. The relation between certain prediction principles to the class of non-Galvin filters (mostly the club filter) on uncountable cardinals has been explored in several papers \cite{GartiTilt,Garti2017WeakDA,bgp,ghhm}.

The aforementioned application of the Borell-Cantelli Lemma implies that, in particular, there are no Dodd-sound\footnote{It was pointed out to us by G. Goldberg that it is possible to prove the non-existence of Dodd-sound ultrafilters on $\omega$ by ultrapower considerations, however, the proof here is different and relies on the Borel-Cantelli Lemma.} ultrafilters on $\omega$ (Corollary \ref{Cor: no Dodd-sound}).

This theorem hints at the possibility that the version of $\diamondsuit^-$ for ultrafilters on $\omega$ may also have implications for their Tukey-order\footnote{In this paper we do not use the Tukey order. However, for completeness reasons, we define the Tukey order for two ultrafilters $U,W$ that $U\leq_T W$ if and only if there is a function $f:W\rightarrow U$ such that for every generating set $\mathcal{B}\subseteq W$, $f''\mathcal{B}$ generates $U$.} types, a topic that has been studied extensively in recent years. In this paper we confirm that this is indeed the case, that is, our diamond-like principles imply being Tukey-top in the usual sense on $\omega$ (Theorem \ref{Thm: diaominf implies Tukey-top}). We omit the definition and background for the Tukey order as it will not be used directly. Instead, we only use the purely combinatorial characterization (condition $(3)$ below) for Tukey-top ultrafilters, which are those ultrafilters, maximal in the Tukey order. 
\begin{theorem}[Isbell]
    Let $U$ be an ultrafilter on $\omega$. Then the following are equivalent:
    \begin{enumerate}
        \item $U$ is Tukey-top i.e. for every ultrafilter $W$ over $\omega$, $W\leq_{\text{T}} U$.
        \item $U\geq_{\text{T}}[\mathfrak{c}]^{<\omega}$.
        \item There is a sequence $\l A_i\mid i<\mathfrak{c}\r\subseteq U$ such that for any infinite $I\in [\mathfrak{c}]^\omega$, $\cap_{i\in I}A_i\notin U$.
    \end{enumerate}
\end{theorem}

Hence the $\diamondsuit^-$ construction presented in this paper provides an alternative construction of Tukey-top ultrafilters in $ZFC$. The first such construction is due to Isbell \cite{Isbell65}, and  
there are several other constructions floating in the literature; Kunen's $\mathfrak{c}$-OK \cite{Kunen1980} for proven to be Tukey top by Milovich in \cite{Milovich08}, and also Dow and Zhou \cite{DowZhou} constructed Noetherian ultrafilters which are also Tukey-top. However, our construction is somehow different than the previous constructions in that it does not involve independent families and the basic ingredient needed to guarantee the $\diamondsuit^-$ is a bit more immune to forcing extensions (Lemma \ref{Lemma: omega-bounding}). The second advantage of our construction is that it has a clean characterization in terms of the ultrapower (Proposition \ref{Prop: Ultrapower formulation}), and that it is $\leq_{RK}$-invariant (Proposition \ref{prop: rk invariant}).

Isbell posed a long-standing open problem regarding the Tukey-top class:

\begin{question}[Isbell]
    Is it consistent that every ultrafilter is Tukey-top? or is there a $ZFC$ construction for a non-Tukey-top ultrafilter?
\end{question}

It is therefore the connection to Tukey-top ultrafilters, and to Isbell's question (rather than the generalization of $\diamondsuit$ to $\omega$) that motivates our investigation of this type of $\diamondsuit$-principles.

In the second part of the paper, we prove that $\diamondsuit^-$ and $\diamondsuit^p$ can be separated (Theorem \ref{Thm: Separating perfect from minus}), namely, one can force an ultrafilter $U$ such that $\diamondsuit^-(U)$ holds but $\diamondsuit^p(U)$ fails. We do not know whether $\diamondsuit^-$ can be separated from Tukey-top.

Finally, we show two applications of this class, the first is the construction of a Tukey-top $q$-point (Theorem \ref{Thm: there is TT-Q-Point}), this is the first such construction ( asked in \cite[p.13 l.-4]{TomNatasha}).  
The second, corresponds to another open problem regarding the Tukey-top class: is the (either Fubini or Cartesian) product of two non-Tukey-top ultrafilters necessarily not Tukey-top? What we show in this paper is that the product of two ultrafilters failing to satisfy $\diamondsuit^-$ necessarily fails to satisfy $\diamondsuit^-$ (Proposition \ref{Prop: non-diamond is close under sum}). In particular, if the Tukey-top class coincides with the $\diamondsuit^-$-class, then the non-Tukey-top class is closed under product.

This paper is organized as follows:

\begin{itemize}
    \item In section~\S 1 we introduce a general principle $\diamondsuit^-(F,\pi,f,T)$ and discuss the limits of those parameters. To strengthen that to the principle $\diamondsuit^-$ we also need the notion of skies. We show that $\diamondsuit^-$ implies Tukey-topness, as well as other properties that will be useful later.
    \item In section~\S 2 we show the existence of $\diamondsuit^-$ in two ways: first using forcing, and then using a construction in ZFC that naturally leads to the principle $\diamondsuit^p$.
    \item In section~\S 3 we show that $\diamondsuit^p$ is strictly stronger than $\diamondsuit^-$.
    \item In section~\S 4 we investigate the preservation of $\diamondsuit^-$ under Fubini product (or more generally Fubini sum).
    \item In section~\S 5 we show the consistency of Tukey-top $q$-points using the principle $\diamondsuit^-$.
    \item In section~\S 6 we list some open questions.
\end{itemize}

\subsection*{Notations \& global assumptions}

Our notations are standard for the most part. A filter $F$ on an infinite set $X$ is a nonempty collection of subsets of $X$ that is closed under intersection and superset, and does not contain $\emptyset$. We shall mostly consider filters on $\omega$, but occasionally allow $X$ to be other countable sets or even arbitrary sets. $F$ is principal if it contains a finite set, and non-principal otherwise. $F^*=\{X\setminus A\mid A\in F\}$ is the dual ideal, and $F^+=\mathcal{P}(X)\setminus F^*$ is the collection of $F$-positive sets. An ultrafilter $U$ is a filter such that for any $A\subseteq X$, either $A\in U$ or $A^c\in U$. If $F$ is a filter on $X$ and $f:X\rightarrow Y$ is a map, then $f_*(F)=\{B\subseteq Y\mid f^{-1}(B)\in X\}$ is also a filter, called the image filter or the pushforward filter. Many properties of $F$ are inherited by $f_*(F)$; for example, if $F$ is an ultrafilter then so is $f_*(F)$. 

For two ultrafilters $U,V$ on $X,Y$ respectively, we say $U$ is Rudin-Keisler reducible to $V$, denoted $U\leq_{RK}V$, if there is a map $f:Y\rightarrow X$ such that $U=f_*(V)$. We call $U$ and $V$ Rudin-Keisler equivalent, denoted $U\equiv_{RK}V$, if there is a bijection $f:X\rightarrow Y$ such that $U=f_*(V)$. It is a standard fact that $U\leq_{RK}V$ and $V\leq_{RK}U$ imply $U\equiv_{RK}V$. 

Let $F$ be a filter on $X$, and $f,g:X\rightarrow\kappa$. We denote by $f\leq_F g$ if $\{x\in X\mid f(x)\leq g(x)\}\in F$ and we say that $f$ is bounded by $g$ mod $F$; variations on this notation such as $f=_F g$ or $f<_F g$ should be self-explanatory. Note that if $f<_F g$ and $F'$ is a filter extending $F$ then $f<_{F'}g$. A function $f$ is bounded mod $F$ if there is $\alpha\in\kappa$ such that $f\leq_F c_\alpha$ where $c_\alpha$ is the constant function $\alpha$. We say that $f$ is unbounded mod $F$ if $f$ is not bounded mod $F$.
Finally, our forcing convention is that $p\leq q$ means $p$ is stronger than $q$.



\section{Diamond-like principle for ultrafilters on $\omega$}
Let us start with a general notion of diamond, which we are quickly going to restrict. Recall that given a filter $F$, $F^*$ is the dual ideal and $F^+$ is the collection of $F$-positive sets.
\begin{definition}
    Let $F$ be a filter over a cardinal $\kappa\geq\omega$, $\pi,f:\kappa\rightarrow\kappa$ be functions and $T\subseteq P(\kappa)$. We say that $\diamondsuit^-(F,\pi,f,T)$ holds if there is a sequence $\l \mathcal{A}_\alpha\mid \alpha<\kappa\r$ such that:
    \begin{enumerate}
        \item $\mathcal{A}_\alpha\subseteq P(f(\alpha))$.
        \item $|\mathcal{A}_\alpha|\leq \pi(\alpha)$.
        \item For every $X\in T$, $\{\alpha<\kappa\mid X\cap f(\alpha)\in \mathcal{A}_\alpha\}$ is $F$-positive.
    \end{enumerate}
    We say that $\diamondsuit^*(F,\pi,f,T)$ holds if $(1),(2)$ above hold and $(3)$ is replaced by
    \begin{enumerate}
        \item [$(3^*)$] For every $X\in T$, $\{\alpha<\kappa\mid X\cap f(\alpha)\in \mathcal{A}_\alpha\}\in F$.
    \end{enumerate}
\end{definition}

The parameter $f$ will often just be the identity map $id$, but it seems necessary to allow general $f$ for the proof of, e.g., Proposition \ref{prop: rk invariant}.

\begin{example}
    \begin{enumerate}
        \item $\diamondsuit^-(\text{Cub}_\kappa,1,id,P(\kappa))$ is just the regular $\diamondsuit(\kappa)$.
        \item $\diamondsuit^-(\text{Cub}_\kappa,id,id,P(\kappa))$ is $\diamondsuit^-(\kappa)$.
        \item 
        As we noted in the introduction, if $f$ is increasing on a set $X\in F$, then $\omega$, $\diamondsuit^-(F,1,f,P(\omega))$ must fail.
        \item $\diamondsuit^*(\text{Cub}_\kappa,id,id,P(\kappa))$ is the usual $\diamondsuit^*(\kappa)$. 
        \item If we  extend the filter $F\subseteq F'$  and the set $T\subseteq T'$, decrease the function $\pi\geq _{F}\pi'$, then $\diamondsuit^-(F',\pi',f,T')\Rightarrow \diamondsuit^-(F,\pi,f,T)$. \item If $U$ is an ultrafilter, then $\diamondsuit^-(U,\pi,f,T)\Leftrightarrow \diamondsuit^*(U,\pi,f,T)$.
    \end{enumerate}
\end{example}
\begin{lemma}
    If there is $f:\kappa\rightarrow\kappa$ such that $\diamondsuit^-(F,\pi,f,T)$ holds, then for every $g\leq_F f$, $\diamondsuit^-(F,\pi,g,T)$.
\end{lemma}
\begin{proof}
    Let $\l\mathcal{A}_\alpha\mid \alpha<\kappa\r$ witness that $\diamondsuit^-(F,\pi,f,T)$ holds. Define $$\mathcal{A}'_\alpha=
        \{X\cap g(\alpha)\mid X\in \mathcal{A}_\alpha\}$$
        Note that if $f(\alpha)\leq g(\alpha)$, then $\mathcal{A}'_\alpha=\mathcal{A}_\alpha$.
    Then we have that for every $\alpha<\kappa$:
    \begin{enumerate}
        \item $\mathcal{A}'_\alpha\subseteq P(g(\alpha))$.
        \item $|\mathcal{A}'_\alpha|\leq |\mathcal{A}_\alpha|\leq \pi(\alpha)$.
    \end{enumerate}
    To see $(3)$, let $X\in T$ be any set, then $B=\{\alpha<\kappa\mid X\cap f(\alpha)\in \mathcal{A}_\alpha\}$ is $F$-positive. By our assumption, $g\leq_F f$, then $C=\{\alpha<\kappa\mid g(\alpha)\leq f(\alpha)\}\in F$. Hence $B\cap C$ is $F$-positive. Let $\alpha\in B\cap C$, then $X\cap f(\alpha)\in \mathcal{A}_\alpha$ and since $g(\alpha)\leq f(\alpha)$, $X\cap g(\alpha)=(X\cap f(\alpha))\cap g(\alpha)\in \mathcal{A}'_\alpha$.
\end{proof}
So the strongest kind of statements we can get is obtained by taking $F$ to be maximal (i.e. an ultrafilter), $T$ being maximal (i.e. $P(\kappa)$) $\pi$ being as small as possible, and $f$ being as large as possible.

The next propositions provide several restrictions on the values of $\pi,f$, and $T$ for which we may expect for the diamond principle to hold. The restriction is derived from the Borel-Cantelli Lemma: Let $(X,\Omega,\mathbb{P})$ be a probability space and $\l E_n\mid n<\omega\r$ be a sequence of events. If $\sum_{n=0}^\infty\mathbb{P}(E_n)<\infty$ then $\mathbb{P}(\limsup E_n)=0$, where $\limsup E_n=\bigcap_{n<\omega}(\bigcup_{m\geq n}E_n)$.

We write $\forall^* n, \ p(n)$ is for all but finitely many $n\in\mathbb{N}$, $p(n)$ holds.

\begin{proposition}
\label{Thm: Borel Cantelli application}
Let $\l\mathcal{A}_n\mid n<\omega\r$ be a sequence of sets such that $\mathcal{A}_n\subseteq P(f(n))$ and $|\mathcal{A}_n|\leq\pi(n)$. Assume that $\sum_{n=0}^{\infty}\frac{\pi(n)}{2^{f(n)}}<\infty$. Then $$\mathbb{P}(\{X\in P(\omega)\mid \forall^*n, \ X\cap f(n)\notin \mathcal{A}_n\})=1$$
where $\mathbb{P}$ is the standard Borel probability measure on $2^\omega$.  
\end{proposition}
\begin{proof}
    Let $E_n$ be the event that $X\cap f(n)\in \mathcal{A}_n$. This is the disjoint union of open sets in $P(\omega)$ and therefore measurable with $\mathbb{P}(E_n)=\frac{|\mathcal{A}_n|}{2^{f(n)}}$. By the Borel-Cantelli Lemma, if $\sum_{n=0}^\infty\mathbb{P}(E_n)<\infty$, then $\mathbb{P}(\limsup E_n)=0$, where $\limsup E_n=\bigcap_{m<\omega}(\bigcup_{n\geq m}E_n)$. It follows that $\mathbb{P}(P(\omega)\setminus \bigcap_{m<\omega}(\cup_{n\geq m}E_n))=1$. It remains to note that $X\notin \bigcap_{m<\omega}(\bigcup_{n\geq m}E_n)$ if and only if there is $m$ such that for all $n\geq m$, $X\notin E_n$.    
\end{proof}
\begin{corollary}\label{cor: borel-cantelli}
    Suppose that $F$ is a filter extending the Fr\'{e}chet filter (i.e. contains all co-finite sets), $\pi,f:\omega\rightarrow\omega$, and $T\subseteq P(\omega)$ are such that $\diamondsuit^-(F,\pi,f,T)$, then:
    \begin{enumerate}
        \item Either $\sum_{n=0}^{\infty}\frac{\pi(n)}{2^{f(n)}}=\infty$, or
        \item $\mathbb{P}(T)=0$, where $\mathbb{P}$ is the standard Borel probability measure on $P(\omega)$.
    \end{enumerate}
\end{corollary}
\begin{proof}
    If $(1)$ fails, then by the previous proposition,
    $$\mathbb{P}(\{X\in P(\omega)\mid \forall^*n, \ X\cap f(n)\notin \mathcal{A}_n\})=1$$
and since $F$ contains the Fr\'{e}chet filter, every $X\in T$ is guessed on an infinite set, and therefore belongs to the complement of the probability one set above. Hence $\mathbb{P}(T)=0$
\end{proof}
Case $(1)$ is quite trivial since we can take for example\footnote{Other examples can be constructed for sets of size $2^{f(n)}-2^{\frac{f(n)}{2}}$, but we do not know if there is a threshold (see Question \ref{Question: Second Borel-Cantelli}).} $\pi(n)=2^{f(n)}$, and then $\mathcal{A}_n$ can be taken to be $P(f(n))$ and $T=P(\omega)$. It is not hard to check that even in case $\pi(n)=2^{f(n)}-1$ we can still take $T=P(\omega)$. Hence it is case $(2)$ which we would like to play with and simply assume that $(1)$ does not happen. If we assume $(1)$ fails, we cannot hope to guess all the subsets of $\omega$ (not even a positive probability subset of $p(\omega)$). 
One more constraint is that if $f$ is not small, than also $\pi$ cannot be too small.
\begin{proposition}
    Suppose that $U$ is an ultrafilter over $\omega$ and $f:\omega\rightarrow\omega$ is unbounded mod $U$. For any $n<\omega$, if $|T|\geq n$ and $\diamondsuit^-(U,\pi,f,T)$ then $\pi$ is not bounded mod $U$ by $n$.   
\end{proposition}
\begin{proof}
    Otherwise, $A=\{m<\omega\mid \pi(m)<n\}\in U$, fix $n$-many distinct sets $X_1,...,X_n$ in $T$, there is $n'<\omega$ such that for every $i\neq j$, $X_i\cap n'\neq X_j\cap n'$. By $\diamondsuit^-(U,\pi,f,T)$, for each $1\leq i\leq n$, $B_i=\{m<\omega\mid X\cap f(m)\in \mathcal{A}_m\}\in U$. Hence $A^*=A\cap(\bigcap_{i=1}^nB_i)\in U$. Since $f$ in unbounded mod $U$, there is $m\in A^*$, such that $f(m)>n'$, but  $X_1\cap f(m),..., X_n\cap f(m)$ are  distinct elements of $\mathcal{A}_m$, and since $m\in B$, $\pi(m)<n$, contradiction.
\end{proof}
The idea above will show up again in Theorem \ref{Thm: diaominf implies Tukey-top}.

Let us impose limitations on the relation between $\pi$ and $f$ which are borrowed from Puritz \cite{PURITZ1972215} terminology of skies and constellations.
Recall that $sky([\pi]_U)<sky([f]_U)$ means $h\circ \pi<_U f$ for any $h:\omega\rightarrow\omega$. 
This is much stronger than the failure of $(1)$, since this, in particular, implies that $\pi(n)<2^{\pi(n)}<f(n)$ on a large set, but much more than that. We will need a generalization of that definition to filters instead of ultrafilters.
\begin{definition}[Puritz]
    Now suppose $F$ is a filter on $\omega$ and $f,g:\omega\rightarrow\omega$ Define $sky([g]_F])< sky([f]_F])$ to mean $\forall h:\omega\rightarrow\omega,\ h\circ g<_F f$, and $sky([g]_F])\geq sky([f]_F])$ to mean $\exists h:\omega\rightarrow\omega,\ h\circ g\geq_F f$.
\end{definition}  
These are both transitive relations, the preorder $\leq$ induces an equivalence relation, and an equivalence class is called \textit{a sky}. There is a greatest element (the ``top sky'') represented by the identity function $\mathrm{id}$, and a least element represented by any constant function. If $F$ is an ultrafilter, then it is a $p$-point if and only if it has only the above two skies. For $f=\mathrm{id}$ we have the following convenient reformulation:

\begin{itemize}
    \item $sky([g]_F)\neq sky([\mathrm{id}]_F)$ iff $\forall X\in F\exists Y\in[X]^\omega\ g\mathord{\upharpoonright}Y$ is constant.
\item $sky([g]_F)< sky([\mathrm{id}]_F)$ iff $\forall X\in F^+\exists Y\in[X]^\omega\ g\mathord{\upharpoonright}Y$ is constant.

\end{itemize}

Of course, these coincide when $F$ is an ultrafilter. We also refer to $sky([g]_F)\neq sky([\mathrm{id}]_F)$ as $g$ is ``not almost one-to-one mod $F$'' or ``does not tend to infinity mod $F$''. In general, we have:
\begin{itemize}
    \item $sky([\pi]_F)\ngeq sky([f]_F)$ iff $\forall X\in F\exists Y\in[X]^\omega\ g\mathord{\upharpoonright}Y$ is constant and $f(Y)$ is infinite.
\item $sky([\pi]_F)< sky([f]_F)$ iff $\forall X\in F^+\exists Y\in[X]^\omega\ g\mathord{\upharpoonright}Y$ is constant and $f(Y)$ is infinite.

\end{itemize}

In \cite{TomGabe}, a connection between $\diamondsuit^-$ and Dodd-sound ultrafilters was established. Let us show that this connection, along with the restrictions obtained by the Borel-Cantelli lemma implies  there are no Dodd-sound ultrafilters on $\omega$. These ultrafilters were introduced by Goldberg \cite{GoldbergUA} which are simplifications of a property due to Steel \cite{SCHIMMERLINGSound} of general elementary embeddings, and have several applications in inner model theory:
\begin{definition}
    A $\kappa$-complete  ultrafilter $U$ over $\kappa\geq\omega$ is called Dodd-sound if there is a sequence $\l \mathcal{A}_\alpha \mid \alpha<\kappa\r$ such that for every sequence of sets $\l X_\alpha\mid \alpha<\kappa\r$, the following are equivalent:
    \begin{enumerate}
        \item $\l X_\alpha\mid \alpha<\kappa\r$ is $U$-threadable (that is, there is $X\subseteq \kappa$ such that $\{\alpha<\kappa\mid X\cap\alpha=X_\alpha\}\in U$).
        \item  $\{\alpha<\kappa\mid X_\alpha\in\mathcal{A}_\alpha\}\in U$.
    \end{enumerate}
\end{definition}
For example, on measurable cardinals, normal ultrafilters are Dodd-sound. Clearly, if $U$ is Dodd-sound then $\diamondsuit^-(U,\pi,id,P(\kappa))$, where $\pi(\alpha)=|\mathcal{A}_\alpha|$.

\begin{corollary}\label{Cor: no Dodd-sound}
    There are no Dodd-sound ultrafilters on $\omega$.
\end{corollary}
\begin{proof}
Suppose otherwise that $U$ is Dodd-sound, then $\diamondsuit^{-}(U,\pi,id,P(\omega))$ holds. By Corollary \ref{cor: borel-cantelli}, we have $\sum_{n=0}^{\infty}\frac{\pi(n)}{2^n}=\infty$. We claim that $B=\{n<\omega\mid \pi(n)>2^{\frac{n}{2}}\cdot n\}\in U$, just otherwise, as $U$ is an ultrafilter, $\{n<\omega\mid \pi(n)\leq 2^{\frac{n}{2}}\cdot n\}\in U$, and changing the valued of $\pi$ on a $U$-null set does not impact the $\diamondsuit^-$, so we may assume that for all $n$, $\pi(n)\leq 2^{\frac{n}{2}}\cdot n$. But then $\sum_{n=0}^{\infty}\frac{\pi(n)}{2^n}\leq\sum_{n=0}^{\infty}\frac{n}{2^{\frac{n}{2}}}<\infty$, contradicting Corollary \ref{cor: borel-cantelli}. Now let us construct a sequence of sets $X_n\in\mathcal{A}_n$ which cannot be $U$-threadable, contradicting $U$ being Dodd-sound.

Suppose that we have constructed $X_i\in \mathcal{A}_i$ for $i<n$. If $n\notin B$, choose $X_n\in\mathcal{A}_n$ randomly. Otherwise, if $n\in B$, we note that there are at most $2^{\frac{n}{2}}\cdot n$-many sets  $Y\in p(n)$ such that $Y\cap \frac{n}{2}=X_i\cap \frac{n}{2}$ for some $i<n$. Since $\pi(n)>2^{\frac{n}{2}}\cdot n$, there is a set $X_n\in \mathcal{A}_n$, such that for every $i<n$, $X_n\cap \frac{n}{2}\neq X_i\cap\frac{n}{2}$. Finally, to see that the sequence $\l X_n\mid n<\omega\r$ is not $U$-threadable, suppose otherwise, and let $X\subseteq \omega$ such that $C=\{n<\omega\mid X\cap n=X_n\}\in U$. Pick any $c_1<c_2\in C\cap B$. Then $X\cap c_1=X_{c_1}$ and $X\cap c_2=X_{c_2}$. However,  $X_{c_2}\cap \frac{c_1}{2}\neq X_{c_1}\cap \frac{c_1}{2}$, this is a contradiction.
\end{proof}
\begin{definition}\label{Def: diamond}
    For an ultrafilter $U$ over $\omega$ and $\lambda\leq \mathfrak{c}$, we say that $\diamondsuit^-_\lambda(U)$ holds if there are sequence $f$,$\pi$,$T$ such that:
    \begin{enumerate}
        \item $sky([\pi]_U)<sky([f]_U)$.
        \item $|T|=\lambda$.
        \item $\diamondsuit^-(U,\pi,f,T)$.
           \end{enumerate} 
           We denote by $\diamondsuit^-(U)=\diamondsuit^-_{\mathfrak{c}}(U)$
\end{definition}
Note that $T,\pi$ are not essential in the definition as we can look at the maximal possible set $T$ which is the set of all $x\in p(\omega)$ such that $\{n<\omega\mid x\cap f(n)\in\mathcal{A}_n\}\in U$, and take $\pi(n)=|\mathcal{A}_n|$. However, we shall keep using $\pi,T$ for clarity reasons.
\begin{corollary}
    If $U$ is an ultrafilter for which there are $T, \l \mathcal{A}_n\mid n<\omega\r$ as in definition \ref{Def: diamond} with respect to $\pi=id$, and $n\mapsto |A_n|$ is not one-to-one mod $U$ but still unbounded, then $\diamondsuit^-_{|T|}(U)$ holds.
\end{corollary}
Given an ultrafilter $U$ over $\omega$ we denote by $M_U=V^\omega/U$, the class of all equivalence class of functions $f:\omega\rightarrow V$ with the relation $E$ being $[f]_UE[g]_U$ iff $\{n<\omega\mid f(n)\in g(n)\}\in U$. $(M_U,E)$ is a (usually not well-founded) model of $ZFC$ and $j_U:V\rightarrow M_U$ defined by $j_U(x)=[c_x]_U$ where $c_x$ is the constant function with value $x$ is an elementary embedding of the universe $V$ into the ultrapower $M_U$. Recall that Lo\'{s} Theorem says that for any first-order formula $\phi(x_1,...,x_n)$ in the language of set theory and any $f_1,...,f_n:\omega\rightarrow V$ 
$$M_U\models \phi([f_1]_U,...,[f_n]_U)\text{ iff } \{m<\omega\mid \phi(f_1(m),...,f_n(m))\}\in U.$$
In terms of the ultrapower by $U$ we can characterize $\diamondsuit^-(U)$ as follows:
\begin{proposition}\label{Prop: Ultrapower formulation}
    Let $U$ be an ultrafilter over $\omega$, then $\diamondsuit^-(U)$ holds iff there is $\mathcal{A}\in M$ and $[f]_U E j_U(\omega)$ such that:
    \begin{enumerate}
        \item $sky(|\mathcal{A}|^{M_U})<sky([f]_U)$.
        \item the set $\{x\in P(\omega)\mid M_U\models j_U(x)\cap [f]_U E\mathcal{A}\}$ has size continuum. Of course,  $``\cap "$ should be interpreted in $M_U$ using $E$.
    \end{enumerate}
\end{proposition}
\begin{proof}
    The translation goes through the ultrapower representation of $\mathcal{A}=[n\mapsto \mathcal{A}_n]_U$, $[\pi]_U=|\mathcal{A}|^{M_U}$ and $T=\{x\in P(\omega)\mid M_U\models j_U(x)\cap [f]_U E\mathcal{A}$. Note that by Lo\'{s} Theorem,  $x\in T$ iff 
    $\{n<\omega\mid x\cap f(n)\in\mathcal{A}_n\}\in U$. Hence $\diamondsuit^-(U)$ is equivalent to the ultrapower formulation in the proposition.
\end{proof}


\begin{theorem}\label{Thm: diaominf implies Tukey-top}
    $\diamondsuit^-_\lambda(U)$ implies the existence of $\l A_\alpha\mid \alpha<\lambda \r\subseteq U$ such that for every $I\in[\lambda]^\omega$, $\bigcap_{i\in I}A_i\notin U$. In particular, $[\lambda]^{<\omega}\leq_T U$.
\end{theorem}
\begin{proof}
For the second part, see for example \cite{Isbell65} or \cite{DobrinenTukeySurvey15}.
    For each $X\in T$, let $B_X=\{n<\omega\mid X\cap f(n)\in A_n\}$. We claim that $\l B_X\mid X\in T\r$ is the desired sequence. Suppose not, then there are distinct sets $X_n\in T$ such that $X=\bigcap_{n<\omega}B_{X_n}\in U$. 

    Since $sky([\pi]_U)<sky([f]_U)$, there is an infinite subset $Y$ of $X$ on which $\pi$ is constant, say with value $N$, and such that $f[Y]$ is infinite. Choose $n\in Y$ for which $f(n)$ is large enough, so that $X_1\cap f(n),\dots,X_{N+1}\cap f(n)$ are all distinct. Since $n\in X\subseteq B_{X_i}$ we have $X_i\cap f(n)\in\mathcal{A}_n$, which implies $|\mathcal{A}_n|>N=\pi(n)$, a contradiction.
\end{proof}
\begin{corollary}
    If $\diamondsuit^-(U)$ holds then $U$ is Tukey-top.
\end{corollary}

In the next section, we shall construct such ultrafilters, and our approach will be similar to the one taken by Isbell; we construct filters with certain properties, guaranteeing that ultrafilters extending these filters are Tukey-top. To emphasize the difference of our filters with Isbell's filters, let us recall how
Isbell used independent families to construct Tukey-top ultrafilters  \cite{Isbell65}.
\begin{definition}
    An independent family is a sequence $\l A_\alpha\mid \alpha<\lambda\r$ such that $A_\alpha\subseteq \omega$ and for every $I,J\in[\lambda]^{<\omega}$ if $I\cap J=\emptyset$ then $$\bigcap_{i\in I}A_\alpha\cap\bigcap_{j\in J}A_j^c\neq\emptyset$$
\end{definition}
Let $\l A_i\mid i<\mathfrak{c}\r$ be an independent family. Such families always exist as proven by Fichtenholz-Kantorevich and Hausdorff \cite{Fichtenholz1934,Hausdorff1936}. Define
$\mathcal{F}^*$ to be the filter generated by the sets
$$\{\bigcap_{i\in I}A_i\setminus \bigcap _{j\in J}A_j\mid I\in[\mathfrak{c}]^{<\omega}, J\in[\mathfrak{c}]^\omega\}$$
Let us call what arises this way from an independent family an \textit{Isbell filter}.
Isbell proved the following:
\begin{lemma}
    Any extension of $\mathcal{F}^*$ to an ultrafilter $U$ is Tukey top as witnessed by the sequence $\l A_i\mid i<\mathfrak{c}\r$.
\end{lemma}
Being an Isbell filter is quite fragile. In particular, any forcing which adds reals (or any $\omega$-sequence in $\mathfrak{c}$), will introduce new $\omega$-sequences of the independent family and therefore the filter might cease being Isbell in the extension.

The filters which we introduce, and give rise to diamond ultrafilters are more robust in this sense.

\begin{definition}
  A filter $F$ over $\omega$ is \textit{good} if there are, $T\in[P(\omega)]^{\mathfrak{c}}$, $f,\pi:\omega\rightarrow\omega$ such that:
\begin{enumerate}
    \item $\diamondsuit^*(F,\pi,f,T)$
        \item For every $g:\omega\rightarrow\omega$, $\{n\in\mathbb{N}\mid g(\pi(n))\geq f(n)\}\in F^+$; this is the same as $sky([\pi]_F)\ngeq sky([f]_F)$.
        
\end{enumerate}
$F$ is called \textit{excellent}, if condition $(2)$ is strengthen to:
\begin{enumerate}
    \item [$(2)^*$] For every $g:\omega\rightarrow\omega$, $\{n<\omega\mid g(\pi(n))<f(n)\}\in F$; this is the same as $sky([\pi]_F)<sky([f]_F)$.
\end{enumerate}
\end{definition} 
Condition $(2)^*$ ensures that for any $F\subseteq F'$, $sky([\pi]_{F'})<sky([f]_{F'})$. Note that in $\diamondsuit^-$, in $(2)$, we only require that the sets are guessed on a positive set, while here we have the stronger  $\diamondsuit^*$ which requires that the guesses occur on a measure-one set in $F$. This guarantees that in any extension of $F$ the sets in $T$ are still guessed.
Since excellent filters already determine all the information needed in the definition of $\diamondsuit^-$, they have a similar property to Isbell filters:
\begin{proposition}
    Suppose that $F$ is an excellent filter, then every ultrafilter $U$ which extends $F$ satisfies $\diamondsuit^-(U)$, and in particular it is Tukey-top.
\end{proposition}

Although goodness need not be preserved under extension, the next lemma shows that this condition is preserved in some situations, which implies that any good filter can be extended to an excellent filter:
\begin{lemma}\label{Lemma: exgtend to an ultrafilter}
    Suppose that $F$ is a good filter as witnessed by $f,\pi:\omega\rightarrow\omega$. Suppose that $g:\omega\rightarrow\omega$ is any function, then  $F[B]$ is good, where  $B=\{n<\omega\mid g(\pi(n))< f(n)\}$.
\end{lemma}
\begin{proof}
    Since $F$ is good, the set $\{n<\omega\mid g(\pi(n))<f(n)\}$ is $F$-positive (by $(3)$), hence $F[B]$ is a legitimate filter. To see that $F[B]$ is still goof, let $h:\omega\rightarrow\omega$ be any function, we need to prove that $C=\{n<\omega\mid h(\pi(n))<f(n)\}$ is $F[B]$-positive. Equivalently, that for every $A\in F$, $C\cap A\cap B\neq\emptyset$. Note that $C\cap B=\{n<\omega\mid g^*(\pi(n))<f(n)\}$, where $g^*(k)=\max(g(k),h(k))$. By $(3)$, this set is $F$-positive and therefore intersect $A$.  
\end{proof}
Since being good is clearly preserved at increasing unions, we get that:
\begin{corollary}
    Every good filter can be extended to an excellent filter and in particular to an ultrafilter $U$ such that $\diamondsuit^-(U)$ holds.
\end{corollary}
When passing from a $ZFC$ model to a generic extension, there are two things that might break down and prevent a good/excellent filter from generating one in the generic extension: the value of $\mathfrak{c}$ might change, and we might add new functions $g:\omega\rightarrow\omega$. The next easy lemma emphasizes how good and excellent filters are less fragile than Isbell filters:
\begin{lemma}\label{Lemma: omega-bounding}
    Suppose that $\mathbb{P}$ is $\omega^\omega$-bounding i.e. every function $f:\omega\rightarrow\omega\in V^{\mathbb{P}}$ is dominated by a ground model function. Also, suppose that the value of $\mathfrak{c}$ is unchanged by $\mathbb{P}$. Then every good/excellent filter $F$ in $V$ generates a good/excellent filter in $V^{\mathbb{P}}$.
\end{lemma}
\begin{proof}
To see condition $(2)/(2)^*$ hold for the generated filter, let $g:\omega\rightarrow\omega\in V[G]$, then there is $h:\omega\rightarrow\omega\in V$ such that for every $n<\omega$, $g(n)<h(n)$. Hence $\{n<\omega\mid h(\pi(n))<f(n)\}\subseteq \{n<\omega\mid g(\pi(n))<f(n)\}$
\end{proof}
Of course, goodness and excellence are preserved even if we require a bit less: that the generic extension is $\omega^\omega$-bounding mod $F$.

In general, let us record the following simple lemma which will be used later:
\begin{lemma}\label{Lemma: good in extension}
    Let $F$ be a good (excellent) filter witnessed by $f,\pi$, and $\mathbb{P}$ be a forcing notion which preserves the value of $\mathfrak{c}$. Let $G\subseteq \mathbb{P}$ be any generic filter and $F\subseteq F'\in V[G]$ be any filter, such that $sky([\pi]_{F'})\not\geq sky([f]_{F'})$ ($sky([\pi]_{F'})<sky([f]_{F'})$), then $F'$ is good (excellent). 
\end{lemma}
\section{On the existence of $\diamondsuit^-$-ultrafilters}
So far we have not proven the existence of an ultrafilter satisfying $\diamondsuit^-(U)$. By the results of the previous section, it suffices to prove the existence of a good filter (which can be extended to an excellent filter, whose extension to any ultrafilter must satisfy $\diamondsuit^-$).
\begin{theorem}
    Adding one Cohen function adds a good filter $F$
\end{theorem}
\begin{proof}
Fix any function $\pi:\omega\rightarrow\omega$ such that for every $n<\omega$, $\pi^{-1}[\{n\}]$ is infinite. Let us consider the following variation of the Cohen forcing. All finite functions $f:\omega\rightarrow [[\omega]^{<\omega}]^{<\omega}$. Namely, each $f(n)$ is a finite collection of finite sets. Let $\l \mathcal{A}_n\mid n<\omega\r$ be $V$-generic. In the generic extension, we define three kinds of sets:
\begin{enumerate}
    \item $X_0=\{n<\omega\mid \mathcal{A}_n\subseteq [P(n)]^{\leq \pi(n)}\}$.
    \item $Y_n=\pi^{-1}[\omega\setminus \{n\}]$.
    \item For each $X\in P^V(\omega)$, let $B_X=\{n<\omega\mid X\cap n\in \mathcal{A}_n\}$
\end{enumerate}
It is not hard to see that the collection $\{X_0\}\cup\{Y_n\}\cup\{B_X\mid X\in P^V(\omega)\}$  generate a filter $F$, and that $\pi$ is not almost one-to-one mod $F$.

\end{proof}

\bigskip
Let us now describe how to construct such ultrafilters in ZFC. A sequence $\langle\mathcal{A}_n\mid n<\omega\rangle$ such that $\mathcal{A}_n\supseteq P(n)$ can be identified with a subset $S$ of the complete binary tree $2^{<\omega}$ (note that $S$ is usually not a subtree). To ensure that continuum many sets are guessed, it suffices that all ``branches'' of $S$ are guessed, as long as $S$ has no ``leaves''. This motivates the perfect diamond principle $\diamondsuit^{p}(U)$ defined below. First, let us fix some tree-related terminology:

A \textit{pseudo-tree} is any subset $S\subseteq 2^{<\omega}$; note that $S$ with the partial order induced from $2^{<\omega}$ is a tree, though not a subtree (as it is not closed under restriction). More generally, for $f:\omega\rightarrow\omega$, we say that $S$ is an $f$-tree if $S$ is a pseudo-tree and there is an identification of every $x\in \mathcal{L}_n(S)$ (where $\mathcal{L}_n(S)$ is the $n^{\text{th}}$ level of $S$) with a function  $t_x\in 2^{f(n)}$ such that  $x<y$ if and only if $t_x,t_y$ are compatible as functions. Note that if $f=id$, and $S$ is a pseudo tree, then we can take $x=t_x$ to see that $S$ is an $id$-tree. For a $f$-tree $S$ and $s\in S$, we say $s$ \textit{branches} in $S$ if there are incompatible $s_1,s_2\in S$ that both extend $s$. Note that in this case it must be that $t_s\subseteq t_{s_1},t_{s_2}$ and $t_{s_1},t_{s_2}$ are incomparable as functions.

\begin{definition}\label{def: perfect-pi-splitting}
    Let $\pi,f:\omega\rightarrow\omega$ be any function. A perfect-$(\pi,f)$-splitting tree is an $f$-tree such that:
    \begin{enumerate}
        \item every $s\in S$ branches in $S$.
        \item $|\mathcal{L}_k(S)|= \pi(k)$.
    \end{enumerate}
    When $f=id$, we omit $S$ and say it is perfect-$\pi$-splitting.


\end{definition}
A \textit{branch through $S$} is any $r\in 2^\omega$ such that $$|\{n<\omega\mid t_x=r\restriction f(n)\text{ for some }x\in \mathcal{L}_n(S)\}|=\aleph_0.$$ We denote by $Br(S)=\{r\in 2^\omega\mid r\text{ is a branch through }S\}$ the set of branches through $S$. 
\begin{remark}
    If $S$ is a perfect-$(\pi,f)$-splitting tree, then $|Br(S)|=\mathfrak{c}$.
\end{remark}
\begin{lemma}\label{lemma: good filter from tree}
    Let $\pi,f:\omega\rightarrow\omega$ be functions and suppose that $S$ is a perfect-$(\pi,f)$-splitting tree such that:
    $$(*) \text{For any finite }F\subseteq Br(S),\ \Big|\Big\{f(m)\ \Big{|} \ \mathcal{L}_m(S)=\{x\mid \exists t\in F(t\restriction f(m)=t_x)\}\Big\}\Big|=\aleph_0.$$
        Then there is a good filter $F$ which guesses $Br(S)$.
\end{lemma}
\begin{proof}
    Let $\mathcal{A}_n=\{\chi_{t_x}\mid x\in\mathcal{L}_n(S)\}$. Then by definition, $\mathcal{A}\subseteq P(f(n))$ and $|\mathcal{A}_n|=|\mathcal{L}_n(S)|=\pi(n)$. For each $r\in Br(S)$ let $B_r=\{n<\omega\mid r\cap f(n)\in \mathcal{A}_n\}$. We claim that the family $$\{B_r\mid r\in Br(S)\}$$
    has the finite intersection property.
    To do this, let us take any  distinct $r_1,\dots,r_d\in Br(S)$, we shall find $n\in \bigcap_{i=1}^dB_{r_i}$.  By $(*)$ applied to $F=\{r_1,...,r_d\}$, 
    there is at least one $m$ such that $r_i\restriction f(m)=t_x$ for some $x\in \mathcal{L}_m(S)$ for all $1\leq i\leq d$. Hence $m\in \bigcap_{i=1}^dB_{r_i}$. Let $F$ be the filter generated by the sets $\{B_r\mid r\in Br(S)\}$. It remains to show that $sky([\pi]_F)\not\geq sky([f]_F)$. Suppose otherwise, that for some $g:\omega\rightarrow\omega$ and some $r_1,...,r_k$ such that for any $n\in \bigcap_{i=1}^kB_{r_i}$, $g(\pi(n))\geq f(n)$. Apply $(*)$ to $\{r_1,...,r_k\}$, then  the set
    $$\Big\{f(m)\ \Big{|} \ \mathcal{L}_m(S)=\{x\mid \exists t\in F(t\restriction f(m)=t_x)\}\Big\}$$
    and therefore also the set
    $$\Big\{f(m)\ \Big{|} \ \mathcal{L}_m(S)=\{x\mid \exists t\in F(t\restriction f(m)=t_x)\}\Big\}\setminus g(k)$$
    is infinite. Take any $f(m)>g(k)$ in that set such that also $r_1\restriction f(m),...,r_k\restriction f(m)$ are all distinct elements of $\mathcal{L}_m(S)$. Then $\pi(m)=\mathcal{L}(S)=k$ and also $m\in\bigcap_{i=1}^kB_{r_i}$. By the choice of $g$, it follows that $$g(k)=g(\pi(m))\geq f(m)>g(k)$$
    contradiction.

\end{proof}
\begin{theorem}\label{Thm: exist splitting tree}
    Let $\pi:\omega\rightarrow\omega$ be any function such that $\{n\mid\pi^{-1}(n)\text{ is infinite}\}$ is infinite. Then there is a perfect-$\pi$-splitting pseudo tree $S$ satisfying $(*)$.
\end{theorem}
\begin{proof}
We shall construct inductively two sequence $\l S_n\mid n<\omega\r$ and $\l m_n\mid n<\omega\r$ such that:
\begin{enumerate}
        \item [(i)] $S_n\subseteq 2^{\leq m_n}$.
        \item [(ii)] if $n\leq m$, then $S_n=S_m\cap 2^{\leq m_n}$ (so $m_n$ is also increasing)
        \item [(iii)] for each $k<m_n$, $S_n$ has at most $\pi(k)$ nodes at level $k$;
\item [(iv)]  any node in $t\in S_{n+1}\setminus S_n$ is above some maximal node of $S_n$ (a node in $S_n$ is maximal if it has no proper extension in $S_n$);

\item [(v)] if $n$ is even then every maximal node of $S_n$ branches in $S_{n+1}$; 

\item [(vi)] if $n$ is odd then none of the maximal nodes of $S_n$ branches in $S_{n+1}$ (so the part of $S_{n+1}$ above that node is a chain), and moreover, for any finite subset $F$ of the set of maximal nodes in $S_n$, there is a level $m_n\leq k<m_{n+1}$ such that $|F|\leq\pi(k)$ and every node in $F$ has an extension in $S_{n+1}$ at level $k$.

    \end{enumerate}
    Once $S_n$ are constructed we let $S=\bigcup_{n<\omega}S_n$. Note that $S$ under the induced order is a perfect tree, and $S_n$ is roughly the part of $S$ up to level $n/2$.  Moreover, by condition $(iv)$, if $x\in Br(S)$, then for any $n$, $x$ is above a unique maximal node in $S_n$. Before turning to the construction of the sequence $S_n$, let us prove that $S$ is as desired:
    \begin{claim}
        $S$ is a perfect-$\pi$-splitting pseudo tree satisfying $(*)$.    \end{claim}
        \begin{proof}[\textit{Proof of claim.}] Condition $(1)$ of Definition \ref{def: perfect-pi-splitting} holds since any $s\in S$ can extended (if necessary) using (vi) to a maximal node in some $S_n$ for some even $n$, then by condition (v) $s$ branches in $S_{n+1}$ and therefore in $S$. For condition $(2)$, we note that if $m_{n-1}\leq k<m_{n}$ (where $m_{-1}=0$), then condition (ii) ensures that $\mathcal{L}_k(S)=\mathcal{L}_k(S_n)$ and condition (iii) ensures that $|\mathcal{L}_k(S_n)|\leq \pi(k)$. To see $(*)$, pick any $x_1,\dots,x_m\in B$, we claim that they are simultaneously guessed at infinitely many levels $k$ with $\pi(k)=m$. Indeed, let $n$ be any odd number large enough so that $x_1,...,x_m$ have distinct (finite) initial segments in $S_n$. If for $i\leq m$ we let $s_i\in S_n$ be the unique maximal node in $S_n$ below $x_i$, then by (vi) there exists $t_i\in S_{n+1}$ above $s_i$ such that $t_i,\ i\leq m$ are at the same level $k$ with $\pi(k)\geq m$. Let $u_i\in S_{n+1}$ be the unique maximal node above $s_i$, so $t_i\sqsubset u_i$. We must have $u_i\sqsubset x_i$ because if we let $u\sqsubset x_i$ be any maximal node in $S_{n+1}$, then $u$ must be $u_i$, just otherwise, $u$ must also be above $s_i$, and since $u,u_i$ are both maximal, they must be incomparable. However, by condition (vi) $s_i$ does not branch in $S_{n+1}$. Hence $t_i\sqsubset u_i\sqsubset x_i$, so the $x_i$'s are simultaneously guessed at level $k$.
            
        \end{proof}
    Let us turn to the construction of $S_n$ and $m_n$. At an even step, for each maximal node $s\in S_n$, find an empty level $k\geq m_n$ above it with $\pi(k)\geq 2$ and add two nodes at level $k$ that extend $s$; let $S_{n+1}$ be $S_n$ together with all these added nodes. Define $m_{n+1}$ to be any number such that $S_{n+1}\subseteq 2^{m_{n+1}}$. At odd steps, enumerate maximal nodes of $S_n$ as $s_1,\dots,s_d$. We find $m_{n+1}$ large enough so that $\pi(m_{n+1})\geq d$, and for every $1\leq i\leq d$ there are $\binom{d}{i}$-many $k\in [m_n,m_{n+1})$ such that $\pi(k)\geq i$. This is possible since $\pi$ is infinite-to-one. Then we choose arbitrary $y_1,\dots,y_d\in 2^{m_{n+1}}$ such that $s_i\sqsubset y_i$. Now for every $F\subseteq\{1,\dots,d\}$, we can allocate a level $k_F\in [m_n,m_{n+1})$ such that $\pi(k_F)=|F|$ and if $F_1\neq F_2$, then $k_{F_1}\neq k_{F_2}$. Define the $\mathcal{L}_{k_F}(S_{n+1})=\{y_i\restriction k_F\mid i\in F\}$. Then clearly $|\mathcal{L}_{k}(S_{n+1})|\leq\pi(k)$, every extension of $s_i$ in $S_{n+1}$ must be a restriction of $y_i$ (this is since $s_i,s_j$, being maximal in $S_n$, cannot end extend each other), hence no maximal branch of $S_n$ branches in $S_{n+1}$, and we made sure that for each set $F$ of maximal branches in $S_n$, there is a level $m_n\leq k<m_{n+1}$ such that $\pi(k)\geq|F|$ and the nodes in $F$ simultaneously extend to the level $k$.

\end{proof}
\begin{corollary}
    There is an ultrafilter $U$ over $\omega$ such that $\diamondsuit^-(U)$ holds.
\end{corollary}
\begin{remark} Theorem \ref{Thm: exist splitting tree} can be  amended to pairs $\pi,f$ such that there are infinitely many $n$'s such that $\{f(m)\mid m\in \pi^{-1}[\{n\}]\}$ is infinite. There is a trivial way to do this: there is a set on which $f$ is injective and $\pi$ is infinite-to-one, and then we can build a perfect-$(\pi,f)$-splitting tree that is essentially just a perfect-$\pi$-splitting tree. To make this interesting at least we want $sky(f)<sky(id)$. But it is unclear whether we can get such an ultrafilter with $sky(f)<sky(id)$. The problem is to maintain both $sky(f)<sky(id)$ and $sky(\pi)<sky(f)$ when extending $F$ to $U$.
\end{remark} 


This construction suggests a different type of diamond:
\begin{definition}\label{Def: Perfect Diamond}
    An ultrafilter $U$ satisfies the \textit{perfect diamond}, denoted by $\diamondsuit^{p}(U)$, if there are $\pi$ infinite-to-one and a perfect-$(f,\pi)$-splitting pseudo tree $S$ such that $\diamondsuit^-(U,\pi,f,Br(S))$.
\end{definition}

We remark that such an $S$ can be identified with a sequence $\l\mathcal{A}_n\mid n<\omega\r$ satisfying $\mathcal{A}_n\subseteq\mathcal{P}(f(n))$, $|\mathcal{A}_n|=\pi(n)$, and whenever $X\in 2^\omega$ is such that $|\{f(n)\mid X\cap f(n)\in\mathcal{A}_n\}|=\aleph_0$, we have $\{n\mid X\cap f(n)\in\mathcal{A}_n\}\in U$.


\section{Separating $\diamondsuit^-$ from $\diamondsuit^p$}
This section is devoted to the investigation of the relation between the following classes of ultrafilter:
\begin{itemize}
    \item $\mathfrak{U}^{\text{top}}=\{U\in \beta\omega\mid [\mathfrak{c}]^{<\omega}\leq_T U\}$.
    \item $\mathfrak{U}^{\diamondsuit^-}=\{U\in \beta\omega\mid\diamondsuit^-(U)\}$.
    \item $\mathfrak{U}^{\diamondsuit^p}=\{U\in \beta\omega\mid \diamondsuit^p(U)\}$
\end{itemize}
First note that $\mathfrak{U}^{\diamondsuit^p}\subseteq\mathfrak{U}^{\diamondsuit^-}\subseteq\mathfrak{U}^{\text{top}}$. The main result of this section is to show that consistently, $\mathfrak{U}^{\diamondsuit^p}\subsetneq\mathfrak{U}^{\diamondsuit^-}$, and therefore also $\mathfrak{U}^{\diamondsuit^p}\subsetneq\mathfrak{U}^{\text{top}}$.

    Let $U\in \mathfrak{U}^{\diamondsuit^-}$ be an ultrafilter. Suppose this is witnessed by $\l\mathcal{A}_n\mid n<\omega\r$, $\pi$,$f$. We define a tree $T_U$ as follows: Let $$B=\{X\in P(\omega)\mid \{n<\omega\mid X\cap f(n)\in\mathcal{A}_n\}\in U\}$$ be the set of all reals which are guessed (Then $|B|=\mathfrak{c}$). We define the $n^{\text{th}}$ level of $T_U$ to be
    $$\mathcal{L}_n(T_U)=\{\l n,\chi_b\r\mid \exists X\in B\exists n<\omega, \ b=X\cap f(n)\in \mathcal{A}_n\}$$
    where $\chi_X$ is the characteristic function for $X$. The order on $T_U$ is defined as follows, $\l n,g\r<\l m,f\r$ if $n<m$ and $f,g$ are compatible as functions.
The following list of some of the properties of $T_U$:
\begin{enumerate}
    \item $|\mathcal{L}_n(T_U)|\leq \pi(n)$.
    \item Every $b\in T_U$ has extensions to unboundedly many levels.
\end{enumerate}
 Replacing $\pi(n)$ with $|\mathcal{A}_n|\leq \pi(n)$, we still obtain a witness for $\diamondsuit^-(U)$, so let us assume that for each $n$, $|\mathcal{L}_n(T_U)|=\pi(n)$. By removing countably many branches from $B$, and shrinking the $\mathcal{A}_n$'s accordingly, we may also assume that:
\begin{enumerate}
    \item [$(2)^*$] Every node $b\in T_U$ branches in $T_U$.
\end{enumerate}
Let us prove the following:
\begin{claim}
    There is $T^*\subseteq T_U$ such that $T^*$ is branching, and apart from countably many $X\in B$, $\{n<\omega\mid \chi_X\restriction f(n)\in \mathcal{L}_n(T^*)\}\in U$. 
\end{claim}
\begin{proof}
    For each $x=\l n_x, b_x\r\in T$ let $$B_x=\{X\in B\mid b_x\sqsubseteq \chi_X\}$$ Let $S$ be the set of all $x\in T_U$ such that $B_x$ is countable. Then $B'=\bigcup_{x\in S}B_x$ is also countable. Let $T^*=T_U\setminus S$. Note that for every $X\in B\setminus B'$, and for any $n$ such that $\l n,\chi_X\restriction f(n)\r\in \mathcal{L}_n(T_U)$, we must have $\l n, \chi_X\restriction f(n)\r\in T^*$. It follows that for each such $X$, $$\{n<\omega\mid \l n,\chi_X\restriction f(n)\r\in \mathcal{L}_n(T^*)\}=\{n<\omega\mid \l n, \chi_X\restriction f(n)\r\in\mathcal{L}_n(T_U)\}\in U.$$

    Let us prove that $T^*$ is branching. Suppose that $x\in \mathcal{L}_m(T^*)$, then $B_x\setminus B'$ is uncountable. Take any distinct $X_1,X_2\in B_x\setminus B'$. Since $X_1,X_2\in B$, there is $n>m$ high enough so that $\chi_{X_1}\restriction f(n)\neq \chi_{X_2}\restriction f(n)$ and $\chi_{X_1}\restriction f(n),\chi_{X_2}\restriction f(n)\in \mathcal{L}_n(T_U)$. Then, again, since $X_1,X_2\notin B'$, $\l n,\chi_{X_1}\restriction f(n)\r,\l n,\chi_{X_2}\restriction f(n)\r\in \mathcal{L}_n(T^*)$ are incomparable extensions of $x$ in $T^*$. 
\end{proof}
Note that $(2)^*$ implies $(2)$. So $(1),(2)^*$ are just saying that $T_U$ is perfect-$(\pi,f)$-splitting tree. This observation leads to an equivalence condition for $\diamondsuit^-(U)$ in terms of trees, which we formulate in proposition \ref{prop: equivalence tree}.
\begin{definition} Let $U$ be an ultrafilter over $\omega$ and $T$ a perfect-$(\pi,f)$-splitting tree. A \textit{$U$-branch through $T$} is any $r:\omega\rightarrow\{0,1\}$ such that $\{n<\omega\mid r\restriction f(n)\in \mathcal{L}_n(T)\}\in U$.
 $T$ is called \textit{$U$-Kurepa} if the set $$B_U(T)=\{r\in 2^\omega\mid r\text{ is a }U\text{-branch in  }T\}$$ has cardinality continuum.
\end{definition}
\begin{proposition}\label{prop: equivalence tree}
    Let $U$ be an ultrafilter over $\omega$, then $U\in\mathfrak{U}^{\diamondsuit^-}$ iff there there are functions $\pi,f$ such that $sky([\pi]_U)<sky([f]_U)$ and a perfect-$(\pi,f)$-$U$-Kurepa tree.
\end{proposition}
\begin{proof}
We already proved that if $\diamondsuit^-(U)$ holds then the tree $T_U$ is a perfect-$(\pi,f)$-$U$-Kurepa tree.  In the other direction, given a perfect-$(\pi,f)$-$U$-Kurepa tree $T$, we define $\mathcal{A}_n=\mathcal{L}_n(T)$, then by definition, $\mathcal{A}_n\subseteq P(f(n))$ and $|\mathcal{A}_n|\leq \pi(n)$. Since the tree is $U$-Kurepa, there are $\mathfrak{c}$-many $U$-branches which is to say that there are $\mathfrak{c}$-many reals which are guessed by $\mathcal{A}_n$.  
\end{proof}

The major difference between $\mathfrak{U}^{\diamondsuit^-}$ and $\mathfrak{U}^{\diamondsuit^p}$, is that the $U$-branches might not include all the branches of the tree $T_U$ i.e. there might be functions $r\in 2^\omega$, such that  $A=\{n<\omega\mid \l n,r\restriction f(n)\r\in \mathcal{L}_n(T_U)\}$ is infinite but $A^c\in U$. 

Note that if $T$ is $U$-Kurepa, then $T$ has the extra property for some set  $B\subseteq Br(T)$ of cardinality continuum:
$$(\dagger) \text{For any finite }F\subseteq B,\ \Big|\Big\{f(m)\ \Big{|} \ \mathcal{L}_m(T)=\{x\mid \exists t\in F(t\restriction f(m)=t_x)\}\Big\}\Big|=\aleph_0.$$
This property classifies the perfect-$(\pi,f)$-splitting trees which give rise to $\diamondsuit^-(U)$ (or equivalently to a $U$ such that $T$ is $U$-Kurepa)
\begin{proposition}
Suppose that $T$ is a perfect-$(\pi,f)$-splitting tree $T$ with a set of branches $B\subseteq Br(T)$ of cardinality $\mathfrak{c}$ with satisfies $(\dagger)$.
Then the filter $F$ generated by $$\Big\{\{n<\omega\mid b\restriction f(n)\in \mathcal{L}_n(T)\}\mid b\in B\Big\}\cup\{\pi^{-1}[\omega\setminus\{n\}]\mid n<\omega\}$$ is good and guesses $B$.
\end{proposition}
\begin{proof}
Compare this with Lemma \ref{lemma: good filter from tree}. The proof is completely analogous.
\end{proof}
To formally establish the distinction of $\mathfrak{U}^{\diamondsuit^-}$ and $\mathfrak{U}^{\diamondsuit^p}$, let us prove that consistently, there exists $U\in \mathfrak{U}^{\diamondsuit^-}\setminus \mathfrak{U}^{\diamondsuit^p}$. The idea is to add new branches to any possible tree and this way ensure that a certain ultrafilter fails to guess all the branches of a tree.
Given a perfect-$(\pi,f)$-splitting tree $T$, let $\mathbb{P}_T$ be the forcing whose underlying set is $T$ ordered by end extension. The following are clear:
\begin{enumerate}
\item $\mathbb{P}_T$ is a non-atomic, separable, c.c.c forcing.
    \item $\mathbb{P}_T$ adds a new branch through $T$
\end{enumerate}
\begin{lemma}\label{Lemma: Successor step of separation}
    Suppose that $T$ is a perfect-$(\tau,f)$-splitting tree, $\pi:\omega\rightarrow\omega$ and $F$ is a filter such that $sky([\tau]_F)\not\geq sky([f]_F)$ and $\pi$ is not almost one-to-one modulo $F$. Suppose that $\pi,\tau$ are comparable modulo $F$, i.e. 
    \begin{enumerate}
        \item $\{n<\omega\mid \tau(n)\leq \pi(n)\}\in F$ or
        \item $\{n<\omega\mid \pi(n)<\tau(n)\}\in F$.
    \end{enumerate}
    Let $b$ be a generic branch for $\mathbb{P}_T$ then $B=\{n<\omega\mid b\restriction f(n)\notin \mathcal{L}_n(T)\}$ is $\bar{F}$-positive, where $\bar{F}$ is the filter generated by $F$ in the generic extension $V[b]$, and $\pi$ is not almost one-to-one modulo $\bar{F}[B]$.
\end{lemma}
\begin{proof}
    Let $p\in T$ be any node and $X\in F$ be any set. By our assumption there is $n$ such that $\{f(m)\mid m\in \tau^{-1}[\{n\}]\cap X\}$ is infinite and in particular $\tau^{-1}[\{n\}]\cap X$ must be infinite. Find any $n+1$ distinct extensions  $p\geq p_1,..., p_{n+1}\in T$ and find $m\in X\cap \tau^{-1}[\{n\}]$ such that $f(m)$ above all the levels of these $p_i$. Since the $m^{\text{th}}$ level of $T$ has only $\tau(m)=n$-many nodes, there must be $1\leq i\leq n+1$ such that $p_i$ has no extension to $m$. Hence $$p\geq p_i\Vdash m\in X\wedge \lusim{b}\restriction f(m)\notin T.$$
    Now the proposition follows by a density argument. Let us prove that $\pi$ is not almost one-to-one modulo $\bar{F}[B]$. Otherwise, there is $p\in T$ and $X\in F$ such that $p\Vdash \pi\restriction X\cap\lusim{B}$ is almost one-to-one. 
    Let us split into cases:
    \begin{enumerate}
        \item [\underline{Case 1:}]If $C=\{n<\omega\mid \tau(n)\leq \pi(n)\}\in F$, then $X\cap C\in F$, and since $\pi$ is not almost one-to-one mod $F$, there is $n$ such that $\pi^{-1}[\{n\}]\cap X\cap C$ is infinite. Let $m<\omega$, and let us perform a density argument as before, let $p\geq q$, find $n+1$-many incomparable extensions $q_1,...,q_{n+1}$ of $q$, find $m'\in X\cap C$, above $m$ and above all the levels of $q_1,...,q_{n+1}$ such that $\pi(m')=n$. Since $m'\in C$, $n=\pi(m')\geq \tau(m')$, and since $T$ is $\tau$-splitting, at level $m'$ there are going to be only $n$-many nodes. in particular, there is $i$ such that $q_i$ had no extension to the $m'$-level. So $q_i\Vdash m'\in X\cap \lusim{B}$ and $\pi(m')=n$. Since $m$ was arbitrary, we reached a contradiction. Before moving to the second case, note that from the assumption that $sky([\pi]_F)\not\geq sky([f]_F)$, we conclude\footnote{Otherwise, there is a set $X$ such that $\pi\restriction X$ is almost one-to-one and then we can define $g(n)=\max(f[\pi^{-1}[n+1]\cap X])$ which is a well-defined function $g:\omega\rightarrow\omega$. Now by the assumption that $\pi\restriction X$ is almost one-to-one. If follows that for every $n\in X$, $n\in \pi^{-1}[\pi(n)+1]\cap X$ and therefore $g(\pi(n))\geq f(n)$, This implies that $sky([\pi]_F)\geq sky)[f]_F)$, contradiction.} that $\tau$ is not almost one-to-one mod $F$, which by the argument we just gave replacing $\pi$ with $\tau$ shows that $\tau$ is also not almost one-to-one with respect to $\bar{F}[B]$.
        \item [\underline{Case 2:}] Suppose $D=\{n<\omega\mid \pi(n)<\tau(n)\}\in F$. Since $\tau$ is not almost one-to-one
        modulo $\bar{F}[B]$, for every $X\in \bar{F}[B]$, there is $n<\omega$ such that $\tau^{-1}[\{n\}]\cap X\cap D$ is infinite. For each $l\in \tau^{-1}[\{n\}]\cap X\cap D$, $\pi(l)<\tau(l)=n$, thus  $\tau^{-1}[\{n\}]\cap X\cap D\subseteq \bigcup_{l<n}X\cap D\cap \pi^{-1}[\{l\}]$. So there must be $l<n$ such that $X\cap \pi^{-1}[\{l\}]$ is infinite. Hence $\pi$ is not almost one-to-one modulo $\bar{F}[B]$.  
    \end{enumerate}
\end{proof}

The idea now is to start with $CH$ and a good filter $F$ which in some forcing extension extends to a good filter, but not perfectly diamond. 

\begin{theorem}\label{Thm: Separating perfect from minus}
    Assume $CH$ and that $F$ is a good filter in $V$. Then there is a fine support iteration of $c.c.c$-forcings of length $\omega_1$ in which $F$ extends to an ultrafilter $U$ which satisfy $\diamondsuit^-(U)$ but not $\diamondsuit^p(U)$.
\end{theorem}
\begin{proof}
    First, let us define the bookkeeping function:
Consider all triples $(\lusim{\pi},\lusim{f},\lusim{T})$ of names for functions $\pi,f:\omega\rightarrow\omega$ and $T$ a perfect-$(\pi,f)$-splitting tree in some finite support of countable length iteration of forcings of the form $\mathbb{P}_{\dot{S}}$. 
Since any such iteration is essentially countable, and any triple $(\pi,f,T)$ can be coded as a subset of $\omega$ (with a fixed function in $V$), there are only $\omega_1$ many such names. Let $h:\omega_1\rightarrow V_{\omega_1}$ be a function such that triple name appears cofinaly many times in $\omega_1$ in the enumeration of $h$.

Let us now define the iteration $\l \mathbb{P}_\alpha,\lusim{Q}_\beta\mid \alpha\leq\omega_1, \ \beta<\omega_1\r$ be a finite support iteration. At stage $\alpha$ of the iteration, we consider the pair $h(\alpha)=(\lusim{\pi}_\alpha,\lusim{f}_\alpha,\lusim{T}_\alpha)$. If it is not a $\mathbb{P}_\beta$-name for some $\beta\leq\alpha$, then we do nothing (namely $\lusim{Q}_\alpha$ is trivial). Otherwise, we let $\lusim{Q}_\alpha$ be $\mathbb{P}_{\lusim{T}_\alpha}$.

Let $F_0\in V$ be a good filter, witnessed by some $\mathcal{A},\pi^*,id,B$. Where $\pi^*$ just not almost one-to-one modulo $F_0$.
 For each $\alpha<\omega_1$, we extend $F_0$ to $F_\alpha\in V^{\mathbb{P}_\alpha}$  such that:
\begin{enumerate}
    \item If $\alpha<\beta$, then $F_\alpha\subseteq F_\beta$.
    \item $F_\alpha$ is continuous for limit $\alpha$.
    \item $\pi^*$ is not almost one-to-one modulo $F_\alpha$.
    
    \item If at stage $\alpha$ of he iteration, either $\lusim{Q}_\alpha$ is trivial, or  $sky([\pi_\alpha]_{F_\alpha})\geq sky([f_\alpha]_{F_\alpha})$, then $F_\alpha=F_{\alpha+1}$.
    \item If at stage $\alpha$ of the iteration $sky([\pi_\alpha]_{F_\alpha})\not\geq sky([f_\alpha]_{F_\alpha})$, and we forced with $\mathbb{P}_{T_\alpha}$, then $$\{n<\omega\mid b_\alpha\restriction n\notin T\}\in F_{\alpha+1}$$ where $b_\alpha$ is the $V^{\mathbb{P}_\alpha}$-generic branch.
\end{enumerate}
The construction of  $F_\alpha$ for a limit $\alpha$ or stages $\alpha+1$ such that either $sky([\pi_\alpha]_{F_\alpha})\geq sky([f]_{F_\alpha})$ or $\lusim{Q}_\alpha$ is trivial, is dictated to us by the requirements $(2),(4)$. Note that $(1)-(5)$ follows for those stages from the inductive assumption for the previous stages (since $\pi^*$ will remain not almost one-to-one at union). If $sky([\pi_\alpha]_{F_\alpha})\not\geq sky([f]_{F_\alpha})$, $\lusim{Q}_\alpha$ is $\mathbb{P}_{T_\alpha}$, and assume inductively that $(1)-(5)$ hold up to and including $\alpha$. Define $F_{\alpha+1}$ as follows consider the set $B_\alpha=\{n<\omega\mid \pi^*(n)<\pi_\alpha(n)\}$. Then $\pi^*$ is not almost one-to-one modulo $F_\alpha[B_\alpha]$ or modulo $F_\alpha[\omega\setminus B_\alpha]$. Let $F'_\alpha$ be the one for which $\pi^*$ is not almost one-to-one modulo $F'_\alpha$, and either $B_\alpha$ or $\omega\setminus B_\alpha\in F'_\alpha$. By the previous Lemma, we can extend $F'_\alpha$ to $F_{\alpha+1}$ so that the $(5)$ holds at $\alpha+1$ and $\pi^*$ is still not almost one-to-one modulo $F_{\alpha+1}$. Hence $(1)-(5)$ still hold.

Let $F^*=\bigcup_{\alpha<\omega_1}F$. Then by Lemma \ref{Lemma: good in extension} $F^*\in V[G]$ is good and by Lemma \ref{Lemma: exgtend to an ultrafilter} in $V[G]$ we can extend\footnote{Alternatively, we could have made sure that $F^*$ is an ultrafilter along the construction of the $F_\alpha$.} $F^*$ to an ultrafilter $U$ such that $\diamondsuit^-(U)$. To see that $\neg\diamondsuit^p(U)$, suppose that $T\in V[G]$ is a perfect-$(\pi,f)$-splitting tree such that $sky([\pi]_U)<sky([f]_U)$. Then since $F^*\subseteq U$, $sky([\pi]_{F^*})\not\geq sky([f]_{F^*})$. By $c.c.c$ of $\mathbb{P}_{\omega_1}$, $(\pi,f,T)$ is added at some stage $\alpha<\omega_1$. Therefore it has a $\mathbb{P}_\alpha$-name that has been enumerated by $h$ at some $\beta>\alpha$. It follows that at the 
$\beta^{\text{th}}$ stage of the iteration, we will have that $\lusim{Q}_\beta$ is interpreted to be $\mathbb{P}_{T}$, and since $F_\beta\subseteq F^*$, we will also have $sky([\pi]_{F_{\beta}})\not\geq sky([f]_{F_\beta})$. Hence $\{m<\omega\mid b_\beta\restriction f(m)\notin\mathcal{L}_m(T)\}\in F_{\beta+1}$ where $b_\beta$ is the generic branch added through $T$. This ensures that $\{m<\omega\mid b_\beta\restriction f(m)\notin\mathcal{L}_m(T)\}\in U$, as $F_{\beta+1}\subseteq U$. We conclude that $\neg\diamondsuit^p(U)$.

\end{proof}



\section{On sums and products of diamond ultrafilters}
Let $U$ be an ultrafilter over a set $X$ and for each $x\in X$, $V_x$ is an ultrafilter over $Y_x$. Recall that the Fubini sum, denoted by $\sum_U V_x$ is the ultrafilter on $\bigcup_{x\in X}\{x\}\times Y_x$ consisting of all $A$ such that $\{x\in X\mid A_x\in V_x\}\in U$, where $A_x=\{y\in Y:(x,y)\in A\}$.
In a more suggestive way, $$A\in \sum_UV_x\text{ iff }\forall^U x\in X\forall^{V_x} y\in Y_x\ (x,y)\in A.$$

If $V_x=V$ (and $Y_x=Y$) for every every $x\in X$, we define the product of $U$ and $V$, denoted by $U\cdot V$ as the ultrafilter $\sum_UV$ over $X\times Y$.

The Tukey type of Fubini products of ultrafilters on $\omega$ was studied first by Dobrinen and Todorcevic in \cite{Dobrinen/Todorcevic11} and Milovich \cite{Milovich12} who proved that for every two ultrafilters\footnote{The order on $V^\omega$ is taken pointwise.} 
$U\cdot V\equiv_T U\cdot V^\omega$. Lately, this investigation was proceeded by Benhamou and Dobrinen \cite{TomNatasha2}, and Benhamou \cite{Commutativity} where the commutativity of cofinal types of Fubini products has been established, and the cofinal type of Fubini sums is analyzed. In this section, we shall prove several closure properties of the ultrafilter classes of interest in this paper, to sums and products. A slight issue is that the Fubini sum of ultrafilters on $\omega$ for example, is an ultrafilter on $\omega\times\omega$, and we did not define what $\diamondsuit^-_\lambda$ means in this case. For a filter $F$ on a set $I$ of cardinality $\kappa$, we say that $\diamondsuit^-(F,\pi,f,T)$ holds if $\pi,f$ are functions from $I$ to $\kappa$, and there is a map $i\mapsto\mathcal{A}_i,\ i\in I$ such that $\mathcal{A}_i\subseteq P(f(i))$ and $|\mathcal{A}_i|\leq\pi(i)$; moreover $T\subseteq P(\kappa)$ is such that for every $X\in T$ we have $\{i\in I\mid X\cap f(i)\in\mathcal{A}_i\}$ is $F$-positive. An ultrafilter $U$ satisfies $\diamondsuit^-_\lambda(U)$ if there are $f,\pi,T$ witnessing $\diamondsuit^-(U,\pi,f,T)$, and moreover $sky([\pi]_U)<sky([f]_U)$ and $|T|=\lambda$. Any ultrafilter on an infinite set $I$ has a Rudin-Keisler isomorphic copy on some cardinal $\kappa\geq\omega$. Hence the following proposition justifies why all of our results in previous sections hold in this more general set-up, and why, for most purposes, it suffices to work with ultrafilter on cardinals.
\begin{proposition}\label{prop: rk invariant}
    If $U$ is an ultrafilter on $I$, $V$ is an ultrafilter on $J$ and $V\leq_{RK} U$, then $\diamondsuit^-_\lambda(V)\Rightarrow\diamondsuit^-_\lambda(U)$. Thus $\diamondsuit^-_\lambda$ is an $\leq_{RK}$-variant.
\end{proposition}
\begin{proof}
    Let $p:I\rightarrow J$ be such that $p_*(U)=V$. If $f$, $\pi$, $\langle X_\alpha\mid\alpha<\lambda\rangle$ and $\langle \mathcal{A}_j\mid j\in J\rangle$ witness $\diamondsuit^-_\lambda(p_*(U))$, then $f\circ p$, $\pi\circ p$, $\langle X_\alpha\mid\alpha<\lambda\rangle$ and $\langle \mathcal{A}_{p(i)}\mid i\in I\rangle$ witness $\diamondsuit^-_\lambda(U)$.
\end{proof}
It follows for example, that if $\diamondsuit^-(U)$, then $\diamondsuit^-(\sum_UU_\alpha)$ (as $U\leq_{RK}\sum_UU_\alpha$).

\begin{remark}
    Note that $\diamondsuit^p$ is also $\leq_{RK}$-invariant. In the notations of the previous propositions, recall that $\diamondsuit^p(V)$ is the same as $\diamondsuit^-(V)$, except the existence of the set $T$ is changed to the following: whenever $X\in 2^\omega$ is such that $|\{f(n)\mid X\cap f(n)\in \mathcal{A}_n\}|=\aleph_0$, we have $\{n\mid X\cap f(n)\in \mathcal{A}_n\}\in V$.  It is clear that if $\{f\circ p(n)\mid X\cap f\circ p(n)\in \mathcal{A}_{p(n)}\}$ is infinite then so must be $\{f(n)\mid X\cap f(n)\in \mathcal{A}_n\}$, so $\{n\mid X\cap f(n)\in \mathcal{A}_n\}\in V=p_*(U)$, which by the definition of Rudin-Keisler reduction means $\{n\mid X\cap f\circ p(n)\in \mathcal{A}_{p(n)}\}\in U$. 
\end{remark}
The next proposition provides a sufficient condition (which will also turn out to be necessary later) for $\sum_UU_n$ to be Tukey-top.
\begin{proposition}\label{Prop simultaiously guessed}
    Suppose that $U,U_\alpha$ are ultrafilters on $\kappa\geq\omega$ and $W=\sum_UU_\alpha$. Assume that $B=\{\alpha<\kappa\mid \diamondsuit^-(U_\alpha)\}\in U$ and for each $\alpha\in B$ $\diamondsuit^-(U_\alpha)$ is witnessed by $(\pi_\alpha,f_\alpha,T_\alpha,\mathcal{A}_{\alpha})$. Then there are $\pi,f:\kappa\times\kappa\rightarrow\kappa$, 
    and $\mathcal{A}$ such that:
    \begin{enumerate}
        \item $sky([\pi]_{W})<sky([f]_{W})$.
        \item $\mathcal{A}_{i,j}\subseteq p(f(i,j))$ and $|\mathcal{A}_{i,j}|\leq \pi(i,j)$.
        \item for every
    $Y\subseteq \kappa$ which is simultaneously guessed by the $U_\alpha$'s i.e.  $\{\alpha<\kappa\mid Y\in T_\alpha\}\in U$, we have
    $$\{(i,j)\in \kappa\times\kappa\mid Y\cap f(i,j)\in \mathcal{A}_{i,j}\}\in W$$.
    \end{enumerate}  
In particular, if there are $2^\kappa$-many $Y$'s which are simultaneously guessed by the $T_\alpha$'s, then $\diamondsuit^-(W)$.
\end{proposition}
\begin{proof}
    Define $\pi,f:I\times J\rightarrow\omega$ by $$\pi(i,j)=\pi_i(j)\text{ and }f(i,j)=f_i(j).$$ Also $\mathcal{A}$ is defined by $\mathcal{A}_{i,j}=(\mathcal{A}_i)_j$. Note that $sky([\pi]_W)<sky([f]_W)$ since for any $h:\kappa\rightarrow\kappa$, $h\circ\pi(i,-)=h\circ\pi_i<_{U_i}f_i$, so $h\circ\pi<_{W}f$ hence $(1)$ holds. Also $(2)$ is clear from the definition of $\pi,f,\mathcal{A}$. Now let $Y\subseteq \kappa$ which is simultaneously guessed by the $T_\alpha$'s.  If $Y\in T_i$, then $\{j<\kappa\mid Y\cap f_i(j)\in(\mathcal{A}_i)_j\}\in U_i$. Therefore, since $Y$ is simultaneously guessed, $\{i<\kappa\mid \{j<\kappa\mid Y\cap f(i,j)\in\mathcal{A}_{i,j}\}\in U_i\}\in U$. This means that $$\{(i,j)\in \kappa\times \kappa\mid Y\cap f(i,j)\in \mathcal{A}_{i,j}\}\in W,$$ as desired.
\end{proof}

\begin{corollary}
    $\diamondsuit^-(U),\diamondsuit^-(V)$ implies $\diamondsuit^-(U\cdot V)$.
\end{corollary}

    To conclude the two results above, if either $\diamondsuit^-(U)$ or $\{\alpha<\kappa\mid\diamondsuit^-(U_\alpha)\}\in U$ and there are $2^\kappa$-many sets which are simultaneously guessed, then $\diamondsuit^-(\sum_UU_\alpha)$ holds. Our next goal is to prove that the converse is also true. Let us start with $\kappa$-complete ultrafilters over measurable cardinals. We will provide two proofs, one uses well-founded ultrapowers and therefore applies only to measurable cardinals $\kappa>\omega$, and the other is purely combinatoric and therefore applies to $\omega$ as well.
\begin{theorem}
    Suppose that $U,U_\alpha$ are $\kappa$-complete ultrafilters over $\kappa$. If $\diamondsuit^{-}(\sum_UU_\alpha)$ holds then either $\diamondsuit^{-}(U)$ holds, or $\{\alpha<\kappa\mid\diamondsuit^{-}(U_\alpha)\}$ holds.
\end{theorem}

\begin{proof}
    Let $W=\sum_UU_\alpha$ satisfy Diamond. This is equivalent to the existence of a set $\mathcal{A}\in M_{W}$, and $\kappa\leq\lambda<j_{W}(\kappa)$ such that
    \begin{enumerate}
        \item $M_{W}\models\mathcal{A}\subseteq P(\lambda)$.
        \item $sky(|\mathcal{A}|^{M_{W}})<sky(\lambda)$, namely $j_{W}(g)(|\mathcal{A}|^{M_{W}})<\lambda$ for any $g:\kappa\rightarrow\kappa$ in $V$.
        \item $|K|=2^\kappa$ where $K=\{X\subseteq \kappa\mid j_{W}(X)\cap \lambda\in \mathcal{A}\}$.
    \end{enumerate}
    Let $V^*=[\alpha\mapsto U_\alpha]\in M_U$, then $M_{W}=(M_{V^*})^{M_U}$ and $j_{W}=j_{V^*}\circ j_U$. 
    Let us split into cases:
    \begin{enumerate}
        \item Case 1: Suppose that $\kappa\leq \lambda<j_U(\kappa)$, then since the critical point of $j_{V^*}$ is $j_U(\kappa)$, $j_{V^*}(\mathcal{A})=\mathcal{A}$, and for every $X\subseteq \kappa$, $j_{W}(X)\cap\lambda= j_U(X)\cap \lambda$. We conclude that  $\mathcal{A},\lambda$ witnesses that $\diamondsuit^-(U)$.
        \item Case 2: Suppose that $j_U(\kappa)\leq\lambda$. In $M_U$, consider the set $\mathcal{B}=\{Y\subseteq j_U(\kappa)\mid j_{V^*}(Y)\cap [id]_{V^*}\in \mathcal{A}\}$. Then since $\mathcal{A},V^*\in M_U$, $\mathcal{B}$ is definable in $M_U$. Note that:
    \begin{enumerate}
        \item $j_U''K\subseteq\mathcal{B}$. 
        \item $\{j_{V^*}(Y)\cap [id]_{V^*}\mid Y\in\mathcal{B}\}\subseteq\mathcal{A}$
    \end{enumerate}
    Let us claim that $M_U\models|\mathcal{B}|=2^{j_U(\kappa)}$. 
    \begin{claim}
        It suffices to prove that for every regular $\kappa<\beta\leq 2^\kappa$, $M_U\models |\mathcal{B}|\geq j_U(\beta)$.
    \end{claim}
    \begin{proof}[Proof of claim] If $2^\kappa$ is regular, then the above is clearly sufficient. Otherwise, if $2^\kappa$ is singular, it is a limit of regular cardinals below it. Since $cf(2^\kappa)>\kappa$, $2^\kappa$ is a continuity point of $j_U$ and therefore $2^{j_U(\kappa)}=\sup\{j_U(\beta)\mid \beta\in 2^\kappa\cap \text{Reg}\}$.
    \end{proof}
Let $\kappa<\beta\leq 2^\kappa$ be regular. Take any $\beta$-many sets $K'=\{k_i\mid i<\beta\}\subseteq K$ and a bijection $\phi:K'\rightarrow \beta$.  By elementarity,  $$M_U\models j_U(\phi):j_U(K')\rightarrow j(\beta) \text{ is a bijection}.$$
    For every $X\in K'$, $j_U(\phi)(j_U(X))=j_U(\phi(X))$. Hence
    $\{j_U(\phi(X))\mid X\in K\}\subseteq j_U(\phi)''[\mathcal{B}\cap j_U(K')]$. Since also $\beta>\kappa$ is regular, it is a continuity point of $j_U$, and therefore $j_U(\phi)''[\mathcal{B}\cap j_U(K')$ is unbounded in $j_U(\beta)$. Since $\beta$ is regular, $j_U(\beta)$ is regular in $M_U$ and it follows that $M_U\models|\mathcal{B}|\geq |\mathcal{B}\cap j_U(K')|=j_U(\beta)$. 
    
    It remains to show that $M_U\models sky^{V^*}(|\mathcal{A}|^{M_{V^*}})<sky^{V^*}(\lambda)$. Note that we only know that $sky^{W}(|\mathcal{A}|^{M_{V^*}})<sky^{W}(\lambda)$. Suppose towards a contradiction that there is $f:j_U(\kappa)\rightarrow j_U(\kappa)\in M_U$ such that $j_{V^*}(f)(|\mathcal{A}|^{M_{V^*}})\geq \lambda$. Then $f=j_U(g)([id]_U)$ for some function $g:\kappa\rightarrow \kappa^\kappa$. Note that since $\lambda\geq j_U(\kappa)$, any two distinct sets $Y,Y'\in \mathcal{B}$ will satisfy that $j_{V^*}(Y)\cap \lambda\neq j_{V^*}(Y)\cap \lambda$ and therefore $|\mathcal{A}|^{M_{V^*}}\geq (2^{j_U(\kappa)})^{M_U}> j_U(\kappa)>[id]_U$. Define 
    $h:\kappa\rightarrow\kappa$ by $h(\alpha)=\sup_{\beta<\alpha}(g(\beta)(\alpha))<\kappa$. Then
    $$j_{W}(h)(|A|^{M_{V^*}})\geq j_{V^*}(j_U(g)([id]_U))(|\mathcal{A}|^{M_{V^*}})=j_{V^*}(f)(|\mathcal{A}|^{M_{V^*}})\geq\lambda.$$ Contradicting the assumption that $sky^{W}(|\mathcal{A}|^{M_{V^*}})<sky^{W}(\lambda)$.
    
    It follows that $M_U\models \diamondsuit^{-}(V^*)$ which by Lo\'{s} Theorem implies $\{\alpha<\kappa\mid \diamondsuit^{-}(U_\alpha)\}\in U$.
    
    \end{enumerate}

\end{proof}
The proof of the previous theorem provides a bit more information:
\begin{corollary}
    Suppose that $U,U_\alpha$ are $\kappa$-complete ultrafilter over $\kappa>\omega$. If $\diamondsuit^{-}(\sum_UU_\alpha)$ and $\lambda$ is the witnessing ordinal then:
    \begin{enumerate}
        \item If $\lambda<j_U(\kappa)$ then $\diamondsuit^{-}(U)$ holds.
        \item If $j_U(\kappa)\leq \lambda$ then $\{\alpha<\kappa\mid \diamondsuit^{-}(U_\alpha)\}\in U$ holds.
    \end{enumerate}
\end{corollary}
The following example shows that we cannot assume in general in the definition of $\diamondsuit^-$ that $f=id$.
\begin{corollary}
    There is an ultrafilter $W$ such that $\diamondsuit^{-}(W)$ holds, but  $[id]_W$ does not witness this.
\end{corollary}
\begin{proof}
    Take any $U$ such that $\diamondsuit^{-}(U)$ and any $p$-point ultrafilter $V$ (then in particular $\diamondsuit^-(V)$ fails) Let $W=U\cdot V$, then $\diamondsuit^-(W)$ holds as $U\leq_{RK} W$. Suppose towards a contradiction that $\lambda=[id]_W$ witness that, note that $[id]_W=[id_{j_U(\kappa)}]_{j_U(V)}=j_U(\kappa)$ (as $j_U(V)$ is normal over $j_U(\kappa)$) and therefore by the previous corollary, we must have that $\diamondsuit^-(V)$ holds,  contradiction.

\end{proof}

The argument below works for every measurable cardinal $\kappa\geq\omega$ and $\kappa$-complete ultrafilters $U,U_\alpha$, but the proof is given only for $\kappa=\omega$ as the other cases were taken care of in the previous theorem.
\begin{proposition}\label{Prop: non-diamond is close under sum} 
    Suppose that $\diamondsuit^-(\sum_UU_n)$ and $T$ is the set of reals which are guessed, then either $\diamondsuit^-(U)$ or $\{n<\omega\mid \diamondsuit^-(U_n)\}\in U$, and every $r\in T$ is simultaneously guessed by the $U_n$'s, namely, there are sequence $\l \mathcal{A}^{(n)}_m\mid m<\omega\r$ and $f_n$ such that $\{n<\omega\mid \{m<\omega \mid r\cap f_n(m)\}\in U_n\}\in U$. 
    \end{proposition}
    \begin{proof}
        Suppose that
         $\diamondsuit^-(\sum_UU_n)$ holds. Then there is $$\l \mathcal{A}_{\l n,m\r}\mid \l n,m\r\in \omega\times\omega\r\text{ and }\pi,f:\omega\times \omega\rightarrow\omega,\text{ and }T\subseteq P(\omega)$$ witnessing this. 
         Let us split into cases:
         \begin{enumerate}
             \item Suppose that there is $Z\in \sum_UU_n$
         such that for every $(i,j),(i,j')\in Z$, $\mathcal{A}_{i,j}=\mathcal{A}_{i,j'}=:\mathcal{A}_i$,  then also $\pi(i,j)$ depends only on $i$ when restricted to $Z$ (as the cardinality of $\mathcal{A}_{i,j}$). We claim that on a measure on set $Z'\in\sum_UU_n$, $f(i,j)$ depends only on $i$. Suppose otherwise, $Z_0=\{i<\omega\mid f_i\text{ is not constant mod }U_i\}\in U$. Let  $X\in T$, then $$Z_X=\{i<\omega\mid \{j<\omega\mid X\cap f(i,j)\in \mathcal{A}_i\}\in U_i\}\in U.$$ Pick any $i\in Z_0\cap Z_X$, then $\{j<\omega\mid X\cap f(i,j)\in \mathcal{A}_i\}\in U_i$, and $f_i$ is non-constant mod $U_i$, it follows that for for unboundedly many $f(i,j)$'s $X\cap f(i,j)\in\mathcal{A}_i$, which implies that $X$ is a finite set, contradicting the cardinality assumption on $T$. Hence on a measure one set in $\sum_UU_n$, $f(i,j)$ depends only on $i$. It follows that $\diamondsuit^-(U)$ must hold.
         \item If for $U$-most $n$ the map $m\mapsto\mathcal{A}_{\l n,m\r}$ is non-constant mod $U_n$, then by a similar argument as in (1), the map $m\mapsto|\mathcal{A}_{\l n,m\r}|$ is non-constant either. For each $n<\omega$ define $\mathcal{A}^n_m=\mathcal{A}_{\l n,m\r}$, $f_n(m)=f(\l n,m\r)$, $\pi_n(m)=\pi(\l n,m\r)$ and $$T_n=\{r\in T\mid \{m<\omega\mid r\cap f_n(m)\in \mathcal{A}^n_m\}\in U_n\}.$$ Then for each $n$, $\pi_n$ and $f_n$ are unbounded mod $U_n$. We claim that the set of all $n<\omega$ such that $|T_n|=\mathfrak{c}$ is in $U$. Otherwise, $X=\{n<\omega\mid |T_n|<\mathfrak{c}\}\in U$. In particular $S=\bigcup_{n\in X}T_n$ has size less than $\mathfrak{c}$ since $\mathrm{cf}(\mathfrak{c})>\omega$. Pick any $r\in T\setminus S$, by definition of $T$, $$Y=\{n<\omega\mid \{m<\omega\mid r\cap f_n(m)\in \mathcal{A}^n_m\}\in U_n\}\in U.$$ Take any $n^*\in Y\cap X$. It follows that $$\{m<\omega\mid r\cap f_{n^*}(m)\in \mathcal{A}^{n^*}_m\}\in U_{n^*},$$ namely $r\in T_{n^*}$, contradicting the choice of $r\notin \bigcup_{n\in X}T_n$.
         So for any $n\in X$ and for each $r\in T_n$, $\{m<\omega\mid r\cap f_n(m)\in \mathcal{A}^n_m\}\in U_n$. It remains to see that $X'=\{n<\omega\mid sky([\pi_n]_{U_n})<sky([f_n]_{U_n})\}\in U$, and then for any $n\in X\cap X'$, $\diamondsuit^-(U_n)$ holds. Otherwise, $\omega\setminus X'\in U$, and for each $n\in \omega\setminus X'$, there is $g_n:\omega\rightarrow\omega$ such that $$\{m<\omega\mid g_n(\pi_n(m))\geq f_n(m)\}\in U_n.$$ Find $g$ such that $g_n\leq^* g$ for each $n\in\omega\setminus X'$. Then for each $n\in \omega\setminus X'$, let $k_n$ be such that for all $k\geq k_n$, $g_n(k)\leq g(k)$. We claim that $\{n<\omega\mid \pi_n\text{ is not bounded mod }U_n\}\in U$. Otherwise $\pi$ depends on $n$ on a measure one set, this is impossible by the argument of proposition 1.6. So there $X''\subseteq\omega\setminus X'$ such that $X''\in U$ and for each $n\in X''$, $Y_n=\{m<\omega\mid \pi_n(m)\geq k_n\}\in U_n$ and for each $m\in Y_n$  such that $g_n(\pi_n(m))\geq f(n)$, $g(\pi_n(m))\geq f(n)$. Hence $\{m<\omega\mid g(\pi_n(m))\geq f_n(m)\}\in U_n$. We conclude that $$\{\l n,m\r\mid g(\pi(n,m))\geq f(\l n,m\r)\}\in \sum_UU_n,$$ contradicting the assumption that $sky([\pi]_U)<sky(f]_U)$. It follows that $\{n<\omega\mid \diamondsuit^-(U_n)\}\in U$ as wanted.
         \end{enumerate}
    \end{proof}

Again, from the proof we can extract a more precise criterion.
\begin{corollary}
    Suppose that $U,U_n$ are ultrafilters on $\omega$, then if $\diamondsuit^{-}(\sum_UU_n)$ with a witnessing sequence $\l\mathcal{A}_{i,j}\mid i,j<\omega\r$, then:
    \begin{enumerate}
        \item if for $U$-almost all  $i$, $j\mapsto\mathcal{A}_{i,j}$ is constant mod $U_i$, then $\diamondsuit^-(U)$ holds. Moreover, for $U$-almost all $i$, $j\mapsto f(i,j)$ is constant mod $U_i$.  
        \item if for $U$-almost all $i$, $j\mapsto\mathcal{A}_{i,j}$ is non-constant mod $U_i$, then $\{n<\omega\mid \diamondsuit^-(U_n)\}\in U$.
    \end{enumerate}
\end{corollary}
Now as before, taking an ultrafilter $W$ on $\omega$ such that $\diamondsuit^-(W)$ and $U$ which does not (for example $U$ can be taken to be a $p$-point), $\diamondsuit^-(W\cdot U)$ cannot be witnessed by taking $f=id$, since this is never constant modulo $U_i$. 

Also we can use these corollaries to try and separate $\mathfrak{U}^{top}$ from $\mathfrak{U}^{\diamondsuit^-}$
\begin{corollary}
    if $\mathfrak{U}^{top}=\mathfrak{U}^{\diamondsuit^-}$ then $\beta\omega\setminus\mathfrak{U}^{top}$ and $\mathfrak{U}^{\diamondsuit^-}$ are closed under sums.
\end{corollary}

\section{Tukey-top $q$-points}
Let us apply the diamond construction to obtain a Tukey-top ultrafilter which is also a $q$-point. Such an ultrafilter was constructed on a measurable cardinal by Gitik and Benhamou in \cite{OnPrikryandCohen}, but the construction uses well-founded ultrapowers and therefore is not adaptable to ultrafilters on $\omega$. This answers a question from \cite{TomNatasha}. Recall that a $U$ is called a \textit{$q$-point} if every almost one-to-one function $f$ mod $U$ is one-to-one mod $U$. Equivalently, for any partition $\l I_n\mid n<\omega\r$ of $\omega$ into finite pieces, there is $X\in U$ such that for every $n$, $|X\cap I_n|\leq 1$, such an $X$ is called a \textit{selector}. One might be tempted to construct such an ultrafilter using Fubini products and sums:
\begin{example}
    Note that for rapid ultrafilter (sometimes called Semi-$q$-points), there is a simple construction, just take a Tukey-top ultrafilter $U$, and take a rapid ultrafilter $V$, then by Miller \cite[Thm. 4]{MillerNoQPoints}, $U\cdot V$ is a Tukey-top rapid ultrafilter. This idea fails for $q$-point as a Fubini product of two ultrafilters on $\omega$ is never a $q$-point as witnessed by the projection to the right coordinate. Sums of the form $\sum_UV_n$ can be $q$-points if for example $U$ is selective (see for example \cite[Cor. 11]{Blass/Dobrinen/Raghavan15}), however, selective ultrafilters are in some sense the opposite of what we would like. One might still try and use sums of different ultrafilters over a Tukey-top ultrafilter, but to say the least, there is no obvious way of doing that.
\end{example}
Note that by Miller \cite{MillerNoQPoints}, it is consistent that there are no $q$-points, hence there is no hope of just constructing such ultrafilter in $ZFC$. Let us give two constructions of such ultrafilters, one uses a simple Cohen forcing for adding $\mathfrak{c}$-many Cohen sets (regardless of the initial value of $\mathfrak{c}$- as long as it is regular). The second is just an absoluteness result which uses the forcing construction and $CH$.  
\begin{theorem}\label{Thm: there is TT-Q-Point}
    (i) If $\mathfrak{c}=\kappa$ is regular, then after adding $\kappa$ many Cohen reals there is a Tukey-top $q$-point. 

    (ii) CH implies there is a Tukey-top $q$-point.
\end{theorem}
\begin{proof}
    (i) We need to ensure that for every partition $\l A_n\mid n<\omega\r$ of $\omega$ into infinitely many finite pieces,  there is a selector $X\in U$, namely $|X\cap A_n|\leq 1$ for all $n$. Start with a model $V$ where $\mathfrak{c}=\kappa$ and there is an ultrafilter $F$ satisfying $\diamondsuit^p$, as witnessed by a perfect-$\pi$-splitting tree $S$. We add $\kappa$ many Cohen reals to get $V[G]=V[\l f_\alpha\mid \alpha<\kappa\r]$, where $f_\alpha$ is the $\alpha^{\text{th}}$-Cohen real. Working in $V[G]$, we inductively extend $F$ to $F_\alpha$ for $\alpha\leq\kappa$, adding one set $B_\alpha$ at each step, in such a way that $F_\kappa$ contains a selector for any finite partition of $\omega$, and also $\pi$ is not almost one-to-one mod $F_\kappa$. In other words, $F_\kappa$ is a good filter, and extends to a Tukey-top $q$-point.
    
    In $V[G]$, enumerate  $\l\vec{A}_\alpha\mid \alpha<\kappa\r$ all the partitions of $\omega$ into finite pieces. At limit steps we take union. At the successor step, suppose that $F_\alpha=F[\l B_\gamma\mid \gamma<\alpha\r]$ (remember that we only add one set at each step). We find $\beta<\kappa$ high enough such that $\l B_\gamma\mid \gamma<\alpha\r,\vec{A}_\alpha,\tau_\alpha\in V[\l f_i\mid i<\beta\r]$ (which exists by c.c.c and regularity of $\kappa$). Let us identify $f_\beta$ with its characteristic set, and let
    $$X=\bigcup \{\vec{A}_{\alpha}(n)\cap f_\beta\mid |\vec{A}_{\alpha}(n)\cap f_\beta|\leq 1, n<\omega\}$$
    Then clearly, $|X\cap \vec{A}_\alpha(n)|\leq 1$ for every $n$. Let us prove that $X\in F_\alpha^+$ and that $\pi$ is still not almost one-to-one mod $F_{\alpha}[X]$. Let $A\in F_\alpha$, we proceed with a density argument. Let $p\in Add(\omega,1)$, we can find $n$ large enough such that for every $m\geq n$ $\vec{A}_\alpha(m)\cap \dom(p)=\emptyset$. Since $\bigcup_{k<n}\vec{A}_\alpha(k)$ is finite, and $\l \vec{A}_\alpha(k)\mid k<\omega\r$ is a partition, there must be some $a\in A\cap \vec{A}_\alpha(m)$ for some $m\geq n$. Extend $p$ to a condition $p'$ such that $p'(a)=1$ and for every $x\in A_m\setminus \{a\}$, $p'(x)=0$. Then $p'\Vdash a\in\dot{X}\cap A$. Hence $X$ is positive with respect to $F_\alpha$. to see that $\pi$ is still not almost one-to-one with respect to $F_\alpha[X]$, suppose otherwise,  there is $A\in F_\alpha$ such that $\pi\restriction A\cap X$ is finite to one. Let $p$ be a condition that forces that (and $A$ now is in the ground $V[\l f_i\mid i<\beta\r]$). Fix any $n$  such that $\pi^{-1}[\{n\}]\cap A$ is infinite. Then we can prove that generically, $X$ intersects this set infinitely many times. This is a contradiction similar to the argument before.

    (ii) Let $S$ be the perfect-$\pi$-splitting tree we constructed previously in Theorem \ref{Thm: exist splitting tree}, and $F$ be the filter generated by sets $B_X$ where $X\in 2^\omega$ is guessed by $S$ and $B_X=\{k\in\omega\mid X\mathord{\upharpoonright}k\in S\}$. It suffices to show the following:

    For any filter $F'$ generated over $F$ by countably many sets such that $\pi$ is not almost one-to-one w.r.t. $F'$, and any partition $\vec{A}$ of $\omega$ into finite sets, $\vec{A}$ has a selector $B$ such that the filter $F'[B]$ is proper, and $\pi$ is still not almost one-to-one w.r.t. $F'[B]$.

    We show this using forcing and absoluteness. Suppose $F'$ is generated over $F$ by $\{B_n\mid n<\omega\}$, and $\vec{A}=\{A_n\mid n<\omega\}$ is a partition of $\omega$ into finite sets. Let $G\subseteq\omega$ be Cohen generic over $V$. Then by the previous argument $B:=\bigcup\{A_n\cap G\mid n<\omega | A_n\cap G|\leq 1\}$ is as desired, except it is in $V[G]$. But consider the statement:

    \textit{there exists} $B$ s.t. $B$ is a selector for $\vec{A}$, and \textit{for any} $X_1,\dots,X_d$ guessed by $S$ and any $n_1,\dots,n_k$, the intersection of $B_{X_1},\dots,B_{X_d},B_{n_1},\dots,B_{n_k},B$ is infinite, and $\pi$ is constant on an infinite subset of this set.

    This is a $\Sigma^1_2$ statement with parameters $A$, $S$, $\{B_n\mid n<\omega\}$ and $\pi$ (coded suitably as reals); note that ``$X$ is guessed by $S$'' and ``$\pi$ is constant on an infinite subset of $B$'' are all expressible with just quantification over natural numbers. Since this statement is true in $V[G]$, by Shoenfield Absoluteness it is true in $V$.
\end{proof}
\section{Open problems}
\begin{question}
Is $\mathfrak{U}^{\text{top}}=\mathfrak{U}^{\diamondsuit^-}$ provable in ZFC?
\end{question}
We conjecture a negative answer, but we also conjecture a positive answer for the following:
\begin{question}
    Is it consistent that $\mathfrak{U}^{\text{top}}=\mathfrak{U}^{\diamondsuit^-}$?
\end{question}
For measurable cardinals, a positive answer was given in \cite{TomGabe}.
\begin{question}
    It is consistent to have a model similar to the one in \cite{TomGabe} for ultrafilters on $\omega$? Namely a model where the following are equivalent:
    \begin{enumerate}
        \item $U$ is non-Tukey-top.
        \item $\neg \diamondsuit^-(U)$.
        \item $U$ is Rudin-Keisler equivalent to an $n$-fold sum of $p$-points.
    \end{enumerate}
\end{question}
The above is not true in general as Blass, Dobrinen and Raghavan proved \cite{Blass/Dobrinen/Raghavan15} that it is consistent to have a non-Tukey-top ultrafilter which is not an $n$-fold sum of $p$-points (not even basically generated- see \cite{Dobrinen/Todorcevic11} for the definition of basically generated ultrafilters). On measurable cardinals, Gitik constructed a similar example \cite{GitikQPoint} (but with a completely different machinery).
\begin{question}
    Are non-Tukey-top ultrafilters closed under Fubini sums? how about Fubini product?
\end{question}
In Proposition \ref{Prop simultaiously guessed}, we showed that for a sequence of ultrafilters $\l U_\alpha\mid\alpha<\kappa\r$, such that $\diamondsuit^-(U_\alpha)$ holds for each $\alpha$ and there are $\mathfrak{c}$-many reals which are simultaneously guessed by the $U_\alpha$'s,  then the sum satisfies $\diamondsuit^-$. 
\begin{question}
Suppose that $\l U_\alpha\mid \alpha<\kappa\r$ is a sequence of $\kappa$-complete ultrafilters over $\kappa$, for $\kappa\geq\omega$, and that for every $\alpha<\kappa$, $\diamondsuit^-(U_\alpha)$. Does it follow that there are continuum many $r$'s which are simultaneously guessed by the $U_\alpha$?
\end{question}
\begin{question}
    Can we characterize $\diamondsuit^p$ in terms of the ultrapower, similar to $\diamondsuit^-$?
\end{question}
\begin{question}
    Is $\diamondsuit^p(U)$ ($\diamondsuit^-(U)$) implies that there is an ultrafilter $W\leq_{RK} U$ such that $\diamondsuit^p(W)$ ($\diamondsuit^-(W)$) is witnessed by $f=id$? 
\end{question}
While on measurable cardinals, it is possible that there is an ultrafilter which guesses every set, the results of this paper \ref{cor: borel-cantelli}, show that full guessing is not possible on $\omega$. Indeed, our principle $\diamondsuit^-$ only requires that continuum many reals are guessed. John Steel asked the following question:
\begin{question}
    What sort of "size" restriction on the set of reals which are guessed by an ultrafilter is consistent? Alternatively, what sort of topological properties are consistent to hold on a set of reals which are guessed by an ultrafilter?
\end{question}

In \cite{Kunen1972}, Kunen constructed $\mathfrak{c}$-OK ultrafilters from an independent family of functions. The class of non $p$-point $\mathfrak{c}$-ok ultrafilters is a subclass of the Tukey-top class as proven in \cite{Milovich08}. Also, it is not hard to see that such ultrafilters must by weak $p$-point (i.e. minimal in the Rudin-F\'{o}lik order) and as a consequence, the sum of ultrafilters is never a $\mathfrak{c}$-OK ultrafilter. In particular, there is a non-$\mathfrak{c}$-OK ultrafilter $U$ such that $\diamondsuit^-(U)$. But the other direction is not clear: 
\begin{question}
    What is the relation between non p-point $\mathfrak{c}$-OK ultrafilters and the $\diamondsuit^-$ principle?
\end{question}
Note that if $U$ is not a $p$-point, then for any countable set $X\subseteq P(\omega)$ and any non-almost one-to-one mod $U$ function $\pi$, there is a tree perfect-$\pi$-splitting tree $T$ with  $X\subseteq Br(T)$ is a set of $U$-branches. Hence a possible strategy to tackle the previous question is to somehow spread out $X$ sets using the $\mathfrak{c}$-OK family to get $\diamondsuit^-(U)$. 

The Tukey-type of ultrafilters ordered by $\supseteq^*$, rather than $\supseteq$ was studied by Milovich \cite{Milovich08}, who proves that Isbell's question can be reformulated using $\supseteq^*$, i.e., he proved that if for every ultrafilter $U$ on omega is Tukey-top under $\supseteq$, then also every ultrafilter is Tukey-top under $\supseteq^*$ (the other direction is trivial). Let us note that Theorem \ref{Thm: diaominf implies Tukey-top} given a bit more:
\begin{theorem}
    $\diamondsuit^-_\lambda(U)$ implies the existence of $\l A_\alpha\mid \alpha<\lambda \r\subseteq U$ such that for every $I\in[\lambda]^\omega$, $\l A_i\mid i\in I\r$ has no pseudo intersection in $U$. In particular, $[\lambda]^{<\omega}\leq_T (U,\supseteq^*)$.
\end{theorem}
\begin{proof}
For the second part, see \cite{Milovich08}.
    For each $X\in T$, let $B_X=\{n<\omega\mid X\cap f(n)\in A_n\}$. We claim that $\l B_X\mid X\in T\r$ is the desired sequence. Suppose not, then there are distinct sets $X_n\in T$ such that for some $X\in U$, $X\subseteq^* B_{X_n}$ for all $n$, namely there is $k_n$ such that $X\setminus k_n\subseteq B_{X_n}$. 

    Find $Y\subseteq X$ on which $\pi$ is constant, say with value $N$, and such that $f[Y]$ is infinite. Choose $n\in Y$ above $k_1,...k_{N+1}$ for which $f(n)$ is large enough, so that $X_1\cap f(n),\dots,X_{N+1}\cap f(n)$ are all distinct. Since $n\in X\setminus k_i\subseteq B_{X_i}$ we have $X_i\cap f(n)\in\mathcal{A}_n$, which implies $|\mathcal{A}_n|>N=\pi(n)$, a contradiction.
\end{proof}
Also note that if $U$ is a $p$-point, then $(U,\supseteq^*)<_T (U,\supseteq)$, since in the former every countable set is bounded.
\begin{question}
Can there be a non $p$-point $U$ such that $(U,\supseteq^*)<_T (U,\supseteq)$? Is it possible that $(U,\supseteq)$ is Tukey-top while $(U,\supseteq^*)$ is not?
\end{question}
Milovich also noted that the property of being Tukey-top is invariant under homeomorphisms of topological spaces and in particular any automorphism of $\beta\omega\setminus \omega$ will move a Tukey-top ultrafilter to a Tukey-top ultrafilter. In attempt to compare $\mathfrak{U}^{\diamondsuit^-}$ with $\mathfrak{U}^{\text{top}}$, it is natural to ask the following:
\begin{question}
Is the class $\mathfrak{U}^{\diamondsuit^-}$ closed under automorphisms of $\beta\omega\setminus \omega$? 
\end{question}
We conclude this paper with a question regarding the existence of a threshold for guessing: 
\begin{question}\label{Question: Second Borel-Cantelli}
    Can we use the second Borell-Cantelli lemma, or some variation of it to say that if $\sum_{n<\omega}\frac{\pi(n)}{2^{f(n)}}=\infty$, then for example the Frechet filter will satisfy $\diamondsuit(F,\pi,f,T)$ for some probability one set $T$?
\end{question}

\bibliographystyle{amsplain}
\bibliography{ref}
\end{document}